\documentclass{article}

\usepackage{graphicx}
\usepackage[utf8]{inputenc}

\usepackage{epsfig}
\usepackage{color}
\usepackage{textcomp}
\usepackage{acronym}
\usepackage[top=1.2in, bottom=1.3in, left=30mm, right=20mm]{geometry}
\usepackage{color}

\usepackage{titling,lipsum}

\usepackage{blindtext, graphicx}
\usepackage{amsmath,amsfonts,amssymb,bbm}
\usepackage{epstopdf}
\usepackage{epsfig}
\graphicspath{{images/}}
\usepackage{mathtools,dsfont}
\usepackage{graphicx}
\usepackage{psfrag,bbm,graphicx}
\usepackage{subcaption}

\usepackage{comment,balance}
\usepackage{epstopdf}
\usepackage{epsfig}
\usepackage{multicol}
\usepackage{multirow, bbm}
\usepackage{amsfonts}
\usepackage{color, cite}
\usepackage{array}
\newtheorem{theorem}{Theorem}
\newtheorem{lemma}{Lemma}

\newtheorem{remark}[theorem]{Remark}
\newtheorem{assumption}{Assumption}
\newenvironment{proof}{{\bf Proof:}}{\hfill\rule{2mm}{2mm}}
\newcommand{\remove}[1]{}

\newcommand{\g}{\gamma}

\newcommand{\pb}{p_{\rm bal}}

\newcommand{\ph}{p_{\rm H}}
\newcommand{\pl}{p_{\rm L}}
\newcommand{\ch}{c(p_{\rm H})}
\newcommand{\cl}{c(p_{\rm L})}
\newcommand{\rl}{R_{\rm L}}
\newcommand{\rh}{R_{\rm H}}
\newcommand{\wl}{w(\gamma p_{\rm L})}
\newcommand{\wh}{w(\gamma p_{\rm H})}
\newcommand{\rc}{\rho_{\rm c}}
\newcommand{\rw}{\rho_{\rm w}}

\definecolor{purple}{rgb}{0,0,0}
\definecolor{blue}{rgb}{0,0,0}
\definecolor{red}{rgb}{0,0,0}

\begin{document}

\title{\vspace{-1in} \bf Duopolistic platform competition for revenue and throughput}
\author{Mansi Sood\thanks{Carnegie Mellon University, Pittsburgh, USA.} \and Ankur A. Kulkarni\thanks{Systems and Control Engineering, Indian Institute of Technology Bombay, Powai, Mumbai 400076.} \and Sharayu Moharir\thanks{Electrical Engineering, Indian Institute of Technology Bombay, Powai, Mumbai 400076.} }
\date{}
\maketitle

\begin{abstract}
We consider two competing platforms operating in a two-sided market and offering identical services to their customers at potentially different prices. The objective of each platform is to maximize its throughput or revenue by suitably pricing its services. We assume that customers have a preference or loyalty to the platforms while the workers freelance for the two platforms. Assuming that the resulting interaction between the users is such that their aggregate utility is maximized, we show that for each value of the loyalty, there exists a pure strategy Nash equilibrium for both the throughput and revenue competition game and characterize it.
\end{abstract}
\vspace{20pt}
\section{Introduction and related work}
The recent past has witnessed the meteoric rise of two-sided marketplaces in domains such as transportation, e.g. Uber, Lyft and the hotel industry, e.g. Airbnb. Platforms like Uber, Lyft, and Airbnb provide a way for the buyer and seller to exchange information about their requirements and thereby perform a transaction. In return for this service, the platform charges a commission for each transaction. The rise in the popularity of these services has led to a growing body to research which focuses on their design and operation \cite{alarabi2016demonstration, banerjee2016multi, chen2015peeking, chen2016dynamic, fang2017prices, stiglic2015benefits, rochet2006two, rysman2009economics, weyl2010price, eisenmann2006strategies} and quantifying their impact \cite{benjaafar2018peer}. 
An important issue that complicates the task of evaluating the impact of such marketplaces is competition among platforms offering similar services. In many cases, two-sided marketplaces are duopolies in which two competing platforms provide near-identical services to users\footnote{In the USA, Uber, and Lyft, and in India, Ola and Uber are rivals in the on-demand transportation business.}. Thus their impact cannot be measured by merely comparing `before' and `after' scenarios with one platform. One must also take into account the influence of competition. 

Our specific focus in this paper is on duopolistic price competition between platforms that provide similar service. We characterize Nash equilibria that emerge from this competition when firms are competing for throughput and for revenue.
In the model we consider, each platform maximizes the \emph{revenue} it generates or the volume of transactions it makes, i.e., its \emph{throughput}. 
In order to maximize its throughput/revenue, each platform incentivises the participation of price-sensitive workers and customers by suitably pricing its services. 
Pricing in two-sided markets is a double-edged sword: while a low price attracts customers, it deters workers. Likewise, a high price deters customers and attracts workers (see e.g., \cite{caillaud2003chicken} where the `chicken and egg' problem is discussed in which a platform must decide which side to incentivise given that each side benefits from increased participation on the other side). In our model, workers freelance for the two platforms and a customer requires one worker for service and a worker can serve at most one customer. Consequently, the number of transactions made via a platform is governed by the user group that sees lower participation at that platform and does not necessarily depend on the total number of users connected to it. In addition, workers/customers have the freedom to choose to provide/avail service at any platform based on the price structure in the market. 

Therefore, a key challenge here is that the price set by one platform affects the user turnout at the other platform. It is therefore of interest to understand what kind of equilibrium pricing structure emerges in this duopoly. We study precisely this question in this paper. 
We model the resulting price competition as a non-cooperative game in which two competing platforms price their services in a manner which maximises their transaction volumes or revenue. The decision variables for the platforms are the prices, which they choose anticipating the outcome of the resulting optimization by the users. Concretely, the objective of each platform is to maximize the throughput/revenue that results from a lower level optimization, which in turn depends on the prices chosen by the platforms.
\begin{figure}[h]
\centering\includegraphics[scale=0.3]{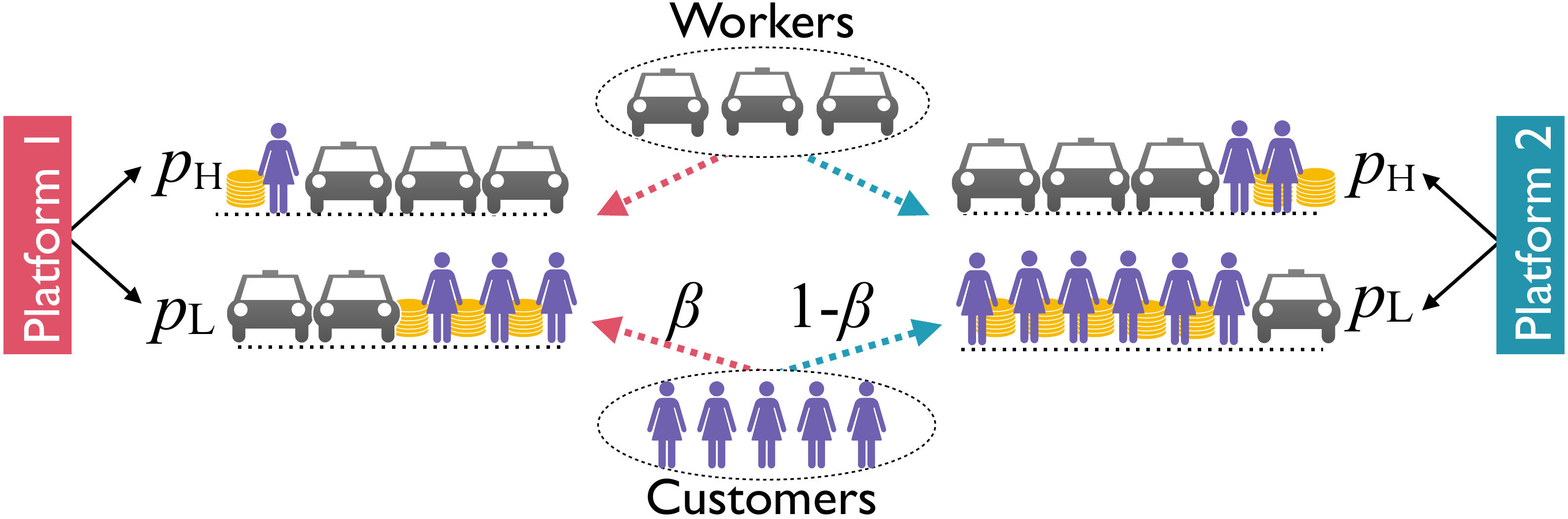}
\caption{Allocation of customers and workers to the platforms.} \label{fig:setting}
\end{figure}

The prices announced by the platforms determine, via demand and supply curves, the total number of customers and workers willing to enter the market at those prices. However, if one platform announces a high price and the other announces a low price, all customers (workers) willing to avail (resp., provide) service at the high (resp., low) price are also willing to avail (resp., provide) it at the low (resp., high) price. Consequently, prices and the associated demand/supply curves do not automatically determine the number of customers or workers joining either platform. {To disambiguate this, we assume that the interaction between workers and customers that results after platforms announce prices leads to an allocation which maximizes the aggregate utility of users.} However, this assumption is not sufficient to ensure uniqueness in the allocation of users to the two platforms given the prices announced by them. This non-uniqueness poses a significant analytical and conceptual challenge (see \cite{kulkarni2013consistency,kulkarni2014shared}) since the platforms set their prices to maximize their throughput/revenue in anticipation of the allocation of users which in turn determines the observed throughput/revenue at the two platforms\footnote{Incidentically, the nonuniqueness is also indicative of why platforms may want to use ``advertising", and other forms of payoff-irrelevant signaling in order to encourage equilibria that are more beneficial to them.}. To make the analysis tractable, we introduce a mechanism for selecting one solution of the user allocation problem described above. We achieve this by factoring in the potential asymmetry in the customer preference for the two platforms. Specifically, each allocation in the set of optimal allocations corresponds to a unique fraction $\beta$ of customers who desire to avail service at Platform 1, as illustrated in Figure \ref{fig:setting}.
We refer to this fixed fraction $\beta$ as the customer \emph{loyalty} and allow it to take values in {$(0,1)$}. We show that a pure strategy {Nash equilibrium for the price competition game for throughput and revenue} exists for each value of the loyalty. {\color{red}In Section 3.6, we generalize our analysis to include price dependent loyalty wherein $\beta$ is a function of prices to model customer preference for lower prices and characterize pure strategy Nash equilibria for throughput and revenue games even in the presence of price dependent loyalty.}

As mentioned before, in the past decade, the task of pricing in two-sided markets has garnered a lot of attention across various communities \cite{rochet2006two, rysman2009economics, weyl2010price, eisenmann2006strategies,comsnets18}. One model~\cite{rochet2003platform} of platform competition consists of users modeled with their private types, using which `quasi demand' curves for the platforms' services are obtained as the probability that a user would like to join a particular platform. This model is more suited to situations where users make long-term subscriptions to platforms they join. We note that our approach is distinct from the above since we assume that users join platforms based on their instantaneous preferences and the market settles at the allocation that maximizes the social welfare of workers and customers. Our approach models situations where users are concerned with the immediate matching of their requests with high disutility for not being matched. The present work significantly extends its conference version~\cite{spcom18} where only throughput competition is considered.

{
The paper is organized as follows. Section~\ref{sec:model} explains the model and Section~\ref{sec:society_routing} explains the {allocation} of users to the platforms. Section~\ref{sec:eq_throughput} and Section~\ref{sec:eq_revenue} provide a characterization of the equilibrium of the throughput and revenue games, respectively. Section~\ref{sec:discussion} presents a discussion on the types of equilibria existing under different market conditions. We summarize our key contributions in Section~\ref{sec:summary} and provide detailed proofs in the Appendix.}

\section{Problem setting}
\label{sec:model}
There are two platforms -- Platform 1 and Platform 2 which provide identical services to their users. In order to {maximize} its throughput (or revenue), each platform chooses one amongst a low price and a high price and advertises this price to its customers. We assume that of the total customers who are served, a fixed fraction ($\beta$, $0 < \beta < 1$) get served at Platform 1. In our setting, each customer's service needs can be satisfied by one worker and each worker can serve at most one customer at a time. We assume that users respond to the prices announced by the platform in a manner that maximises their aggregate utility. The platforms then match customers and workers allocated to them. Corresponding to every payment made by a customer, each platform retains a fixed fraction ($1-\gamma$, $0 < \gamma < 1$) and passes on the remaining amount to the worker who serves that customer. {For analytical tractability, we assume that each platform retains the same fraction $(\gamma)$ per payment.}

\subsection{Supply and demand}
We use the following notation to quantify the availability of freelance workers and customers.
\begin{enumerate}
		\item[1.] $w(p)$: Number of workers willing to work at price $p$. 	
		\item[2.] $c(p)$: Number of customers willing to pay price $p$. 
		\item[3.] $p_i$: Price paid by a customer for service at Platform $i$,
		\item[4.] $c_i$: Number of customers joining Platform $i$,
		\item[5.]  $w_i$: Number of workers joining Platform $i,$
		\\where $i=1,2$.
\end{enumerate}
In each transaction at Platform $i$, a customer pays $p_i$ to Platform $i$ and a worker receives a share $\gamma p_i$ from the transaction, where $i=1,2$.
We make the following assumption on the nature of supply/demand curves and the prices announced by the platforms.
\begin{assumption}[Supply and demand]\label{ass:demandsuppply}{\color{white}a}
\begin{enumerate}
\item[a.] Demand $c(p)$ is a strictly decreasing function of price $p$.
\item[b.] Supply $w(p)$ is a strictly increasing function of $p$.
\item[c.] Prices $p_1$ and $p_2$ can take values amongst a high price $\ph$ or a low price $\pl$, where $\pl < \ph$. 
\item[d.] There exists a unique market clearing price $\pb$ such that $w(\gamma \pb)=c(\pb)$.
\end{enumerate}
\end{assumption}
{Note that we focus on a setting where the supply/demand are completely determined by how the platforms price their services and consequently the supply/demand do not depend on the participation from the other side. Such a situation would persist in service critical applications where the customers' utility only depends on whether or not they get served regardless of how many potential workers are available.}
\subsection{Platform's objective}
We consider two platform objectives -- throughput maximization and revenue maximization. 
The throughput of a platform is given by the minimum of the following two quantities -- the number of workers and the number of customers joining it. When Platform 1 plays $k$ and Platform 2 plays $\ell$, the \emph{throughput} $N_i^{k\ell}$ of Platform $i$ is $\min\{c_i, w_i\}$ and the \emph{revenue} of Platform $i$ is given by $(1-\g)p_i N_i^{k\ell}$, where $i\in \{1,2\}$ and $k,\ell \in \{ \rm L,H\}.$ The platforms price their services in a manner that maximizes their throughput or revenue.

\subsection{Users' objective}
We assume that the workers and customers get matched through the platforms in a manner 
that maximizes their aggregate utility. In order to characterize the utility of the users, we define the following quantities. 
\begin{enumerate}
\item[1.] $b_{w_i}( \gamma p_i)$: Benefit of a worker from serving a customer through Platform $i$, 
\item[3.] $b_{c_i}( p_i)$: Benefit of a customer getting served by a worker through Platform $i$,
\item[2.] $m$: Cost of maintenance for a worker per transaction,
\item[4.] $-d_w$: Disutility experienced by an unmatched worker, 
\item[5.] $-d_c$: Disutility experienced by an unmatched customer, 
\item[6.] $u_i=u_i(p_i):=b_{w_i}( \gamma p_i)+b_{c_i}( p_i)-m-p_i$, where $i=1, 2.$. This is the benefit to the users per transaction at Platform $i$ when it announces a price $p_i$. This includes the utility of a worker-customer pair interacting at Platform $i$, where $i=1, 2.$
\end{enumerate}
We make the following assumptions in describing the utility of users.
\begin{assumption}[Utility of users]\label{ass:utility}{\color{white}a}
\begin{enumerate}
\item[a.] {$p_1, p_2, m, d_w, d_c >0$.}
\item[b.] Price $p_i$ is such that utility of a worker or a customer joining the platform is positive, i.e., $b_{w_i}( \gamma p_i)-m > 0$ and $b_{c_i}( p_i)-p_i > 0$, where $i=1, 2.$ 
\end{enumerate}
\end{assumption}
In the absence of Assumption~\ref{ass:utility}(b), individual workers and customers would not have any incentive to join the platform. 

\section{Allocation of users}%
\label{sec:society_routing}
\label{sec:routing} 
{Recall that once the two platforms set their prices, users are allocated to the platforms in a manner which maximizes their aggregate utility subject to the following constraints.} 
\begin{enumerate}
	\item[a)]  The number of workers/customers joining a particular platform cannot exceed the supply/demand level corresponding to the price set by it. 
	\item[b)] The total number of workers/customers joining either of the platforms can be at most equal to the supply/demand corresponding to the higher/lower price. 
\end{enumerate} The preceding constraints are insufficient to ensure the uniqueness of aggregate utility-maximizing allocation. Hence, we additionally impose the following additional constraint. Recall that $\beta  $ denotes the loyalty.
\begin{enumerate}
	\item[c)] If $c_1$ and $c_2$ customers join Platform 1 and Platform 2 respectively, then the ratio $c_1/c_2=\beta/(1-\beta)$.
\end{enumerate}


\begin{align*}
\underset{w_1,w_2, c_1, c_2} 
\max \ & \ u_1\min\{c_1, w_1\}+u_2 \min\{c_2, w_2\}\\
&-d_w (w_1-\min\{c_1, w_1\}+w_2-\min\{c_2, w_2\}) \\
&-d_c (c_1-\min\{c_1, w_1\}+c_2-\min\{c_2, w_2\}) \qquad  \hfill{\text{(S)}}\nonumber \\
\text{s.t}. \quad & 0 \leq w_1   \leq   w(\gamma p_1), \\
					 &0 \leq w_2  \leq   w(\gamma p_2), \\
					 &w_1+w_2   \leq  w(\max\{\gamma p_1, \gamma p_2\}),\\
					 &0 \leq c_1  \leq   c(p_1), \\
					 &0 \leq c_2  \leq   c(p_2), \\
					 &c_1+c_2   \leq  c(\min\{p_1,p_2\}),	\\								
					&c_1  =  \beta( c_1+c_2).
\end{align*}

In a solution $w_1,w_2, c_1, c_2$ of the above optimization, $w_i,c_i$ respectively denote the number of workers and customers that join platform $i$ when prices $p_1,p_2$ are announced.
 \begin{lemma}\label{lem:lemma1}
 The optimal {solution} to problem (S) is achieved at a point \\$(w_1^*,w_2^*,c_1^*,c_2^*)$ for which $c_1^* \geq w_1^*$ and $c_2^*\geq w_2^*$, with equality holding in at least one case. 
  \end{lemma}
 \begin{proof}
Suppose the claim is not true, i.e., either $c_1^* < w_1^*$ or $c_2^* < w_2^*$ or $c_1^* > w_1^*$ and $c_2^* > w_2^*$. We separately analyse these cases,
\begin{enumerate}
\item[1.] If $c_1^* < w_1^*$, then note that taking $w_1'=c_1^*$, the point $(w_1',w_2^*,c_1^*,c_2^*)$ is feasible since $w_1'=c_1^*<w_1^*\leq w(\gamma p)$ and $w_1'+w_2^*=c_1^*+w_2^*<w_1^*+w_2^*\leq w(\gamma p)$. Moreover, the objective value increases by $d_w (w_1^*-c_1^*)$. 
\item[2.] For the case $c_2^* < w_2^*$, as above we can get a higher objective value by taking $w_2$ as $c_2^*$.
\item[3.] Finally, when $c_1^* > w_1^*$ and $c_2^* > w_2^*$, we can decrease $c_1^*, \ c_2^*$ proportionately to get a larger objective value while ensuring feasibility is retained. Suppose, for some $x>0$, we have $c_1^*=\beta x, \ c_2^*=(1-\beta) x$. Now, there exists a small positive $\Delta$ such that $c_1=\beta (x-\Delta)> w_1^*, \ c_2=(1-\beta) (x-\Delta)> w_2^*$ is feasible and admits a higher objective value.
\end{enumerate}
Thus in all the three cases, $(w_1^*,w_2^*,c_1^*,c_2^*)$ is sub-optimal which gives a contradiction and hence Lemma \ref{lem:lemma1} holds true. \end{proof}

We proceed with evaluating the throughput/revenue for various combinations of prices announced by the platform.
Using Lemma \ref{lem:lemma1}, the objective can be rewritten as $u_1 w_1+u_2 w_2 -d_c (c_1-w_1)-d_c(c_2-w_2)$. {For the purpose of analysing the equilibrium strategies, we make the following assumption.}
{\begin{assumption}[Disutility of customers]\label{ass:utility2}
We assume that the disutility of an unmatched customer is such that $d_c >\frac{\beta}{1-\beta}u_1$ and $d_c>\frac{1-\beta}{\beta}u_2$. 
\end{assumption}}
This is motivated by settings in which the requested service is critical and the customer is at a great loss if left unserved. 
We work with Assumption~\ref{ass:utility2} for the remainder of our analysis.
\begin{lemma} \label{lem:lemma2}
If Assumption \ref{ass:utility2} holds, then at the social welfare optimal point \\$(w_1^*,w_2^*,c_1^*,c_2^*)$ there are no unmatched customers or workers, i.e., $w_1^*=c_1^*$ and $w_2^*=c_2^*$. 
\end{lemma}
\begin{proof}
From Lemma \ref{lem:lemma1}, we can have the following two cases, $c_2^*=w_2^*$, $c_1^* \geq w_1^*$ or $c_1^*=w_1^*$, $c_2^* \geq w_2^*$.
We first analyse the case $c_2^*=w_2^*$, $c_1^* \geq w_1^*$.
The optimization problem becomes,\\
\begin{align*}
\underset{w_1, c_1}
\max \ & (u_1+d_c)w_1 +(\frac{1-\beta}{\beta}u_2 -d_c )c_1\nonumber \\
\text{s.t.} \quad & 0 \leq w_1   \leq   \min\{c_1,w(\gamma p_1)\}, \\
					& w_1+\frac{1-\beta}{\beta}c_1   \leq w(\max\{\gamma p_1, \gamma p_2\}),\\	
					& c_1 \leq \min\{ \frac{\beta}{1-\beta} w(\gamma p_2), c(p_1), \frac{\beta}{1-\beta}c(p_2),   \beta c(\min\{p_1,p_2\}) \}.
\end{align*}
Suppose $w_1^*\neq c_1^*$. From Lemma \ref{lem:lemma1}, it follows that $w_1^* < c_1^*$. Now, consider $\Delta>0$ such that $w_1^*<c_1^*-\Delta<c_1^*$. Note that in the objective function, the coefficient of $c_1$ is negative and taking $c_1$ as $c_1^*-\Delta$ retains feasibility and admits a higher objective value. This violates the optimality of $(w_1^*,w_2^*,c_1^*,c_2^*)$ and hence $w_1^*\neq c_1^*$. Next, we analyse the case $c_1^*=w_1^*$, $c_2^* \geq w_2^*$. Using Lemma \ref{lem:lemma1} simplifies (S) to the following equivalent optimization problem.

\begin{align*}
\underset{w_2, c_2}
\max \ & (u_2 +d_c)w_2+( \frac{\beta}{1-\beta}u_1-d_c )c_2\nonumber \\
\text{s.t.} \quad & 0 \leq w_2   \leq    \min\{c_2,w(\gamma p_2)\}, \\
					& \frac{\beta}{1-\beta}c_2 +w_2   \leq w(\max\{\gamma p_1, \gamma p_2\}),\\	
					& c_2 \leq \min\{ \frac{1-\beta}{\beta} w(\gamma p_1), c(p_2), \frac{1-\beta}{\beta}c(p_1),   (1-\beta) c(\min\{p_1,p_2\}) \}.
\end{align*}
Arguing as above, we can show that a point $(w_1^*,w_2^*,c_1^*,c_2^*)$ such that $w_2^* < c_2^*$ is sub-optimal and hence $c_2^* =w_2^*$.
\end{proof}
\begin{theorem}[Routing of users]\label{thm:routing}
The solution to the social welfare maximization problem (S) is the following,
\begin{align*}
&c_1^*=w_1^*= \min \Big \{   w(\gamma p_1), \frac{\beta}{1-\beta} w(\gamma p_2), c(p_1), \frac{\beta}{1-\beta}c(p_2), \\
& \qquad \qquad \beta c(\min\{p_1,p_2\}), \beta w(\max\{\gamma p_1, \gamma p_2\}) \Big \},\\
&c_2^*=w_2^*= \min \Big \{ \frac{1-\beta}{\beta} w(\gamma p_1),  w(\gamma p_2), \frac{1-\beta}{\beta}c(p_1), c(p_2), \\
& \qquad \qquad (1-\beta) c(\min\{p_1,p_2\}), (1-\beta) w(\max\{\gamma p_1, \gamma p_2\})\Big  \}.
\end{align*}
\end{theorem}
\begin{proof}
From Lemma \ref{lem:lemma2}, we get the following equivalent optimization problem for (S).\\
\begin{align*}
\underset{c_1}
\max \ & (u_1 + \frac{1-\beta}{\beta}u_2)c_1\nonumber \\
\text{s.t.} \quad & 0 \leq c_1  \leq \min\{ w(\gamma p_1), \frac{\beta}{1-\beta} w(\gamma p_2), c(p_1), \frac{\beta}{1-\beta}c(p_2), \\
					&  \qquad \qquad \beta c(\min\{p_1,p_2\}), \beta w(\max\{\gamma p_1, \gamma p_2\}) \}.
\end{align*}
Note that in the objective function, the coefficient of $c_1$ is positive and thus the largest feasible $c_1$ is optimal. The values of $c_2,w_1,w_2$ are then obtained using Lemma~\ref{lem:lemma2}.
\end{proof}
{
Thus, if Assumptions~\ref{ass:demandsuppply},~\ref{ass:utility}~and \ref{ass:utility2} hold, there is a unique solution to Problem (S). Thus, for the optimal allocation problem in the \textit{absence} of loyalty constraint, each of the possible multiple solutions corresponds to a unique loyalty $c_1/(c_1+c_2)$; hence 
$\beta$ can be interpreted as a mechanism to pick a unique solution from the solution set of this problem.}

\section{Equilibrium strategies for throughput competition }\label{sec:eq_throughput} 
In this section, we analyse the non-cooperative game between the platforms as they compete for throughput. The two platforms can choose to play either a low price $\pl$ or a high price $\ph$ (these strategies are denoted $\rm L$ and $\rm H$). After the prices have been announced, the number of workers and customers joining the two platforms is given by Theorem~\ref{thm:routing}. This gives the throughput of the two platforms for any pair of strategies chosen by the platforms. Thus, we get a game in which each platform seeks to maximize its throughput. We denote the payoff (throughput) of Platform $i$ when Platform 1 plays $k$ and Platform 2 plays $\ell$ by $N_i^{k\ell}$, where $i\in \{1,2\}$ and $k,\ell \in \{ \rm L,H\}.$ 
The payoff matrix for the resulting game is given in Table~\ref{tab:payoff_structure}. Recall that from Theorem~\ref{thm:routing}, it follows that the throughput for Platform 2 is same as the throughput of Platform 1 scaled by factor $ \frac{1-\beta}{\beta}$.
\begin{table}[h]
	\begin{center}
		\setlength{\extrarowheight}{2pt}
		\begin{tabular}{cc|c|c|}
			& \multicolumn{1}{c}{} & \multicolumn{2}{c}{Platform $2$}\\
			& \multicolumn{1}{c}{} & \multicolumn{1}{c}{$\rm L$} & \multicolumn{1}{c}{$\rm H$} \\\cline{3-4}
			\multirow{2}*{Platform $1$} & $\rm L$ & $(N_1^{\rm LL},N_2^{\rm LL})$ & $(N_1^{\rm LH},N_2^{\rm LH})$ \\\cline{3-4}
			& $\rm H$ & $(N_1^{\rm HL},N_2^{\rm HL})$ & $(N_1^{\rm HH},N_2^{\rm HH})$ \\\cline{3-4}
		\end{tabular}
		\caption{Payoff matrix for the two platforms. }\label{tab:payoff_structure}
	\end{center}
\end{table}
From Theorem~\ref{thm:routing}, we obtain the following expression for the throughput of Platform 1 for different strategies chosen by the platforms. 
\begin{align}
N_1^{\rm LL} &= \min\{\beta c(\pl),\beta w( \gamma \pl) \},\label{eq:nll}\\
N_1^{\rm LH} &= \min \{\dfrac{\beta}{1-\beta}c(\ph), \beta c(\pl), w(\gamma \pl), \beta w( \gamma \ph) \},\label{eq:nlh}\\
N_1^{\rm HH} &= \min\{\beta c(\ph),\beta w( \gamma \ph) \},\label{eq:nhl}\\
N_1^{\rm HL} &= \min \{c(\ph), \beta c(\pl), \dfrac{\beta}{1-\beta}w(\gamma \pl), \beta w( \gamma \ph) \}\label{eq:nhh}.
\end{align}
Next, we characterize the Nash equilibria in pure strategies for this game and ascertain the impact of the loyalty $(\beta)$ on the same.
Depending on whether the demand and supply match at the prices $\ph$ or $\pl$, we separately analyse the following three cases -- 

\begin{enumerate}
	\item[1.] $\pl < \pb< \ph$.
	\item[2.] $\pl < \ph \leq \pb$.
	\item[3.] $\pb \leq \pl< \ph$.
\end{enumerate}

\subsection{$\pl < \pb< \ph$}
In this case, the market is strictly supply limited at $\pl$ and strictly demand limited at $\ph$. The succeeding theorem gives the equilibrium strategies for the platforms. From Assumption \ref{ass:demandsuppply}, we have $\beta w( \gamma \pl) \leq \beta c(\pl)$ and $\beta c(\ph) \leq \beta w( \gamma \ph).$ Therefore, $N_1^{\rm LL} = \beta w(\gamma \pl)$ and $N_1^{\rm HH} = \beta c(\ph)$. We find that the nature of the Nash equilbrium depends intimately on the ratio,
\begin{equation}
\rho:=\dfrac{\ch}{\wl}.
\end{equation}
\begin{theorem}[Equilibrium strategies]\label{thm:eq1}
	Let Assumption~\ref{ass:demandsuppply} hold and let $\pl < \pb< \ph$.
	Then the pure strategy Nash equilibria (NE) for the throughput maximization game between the two competing platforms are as follows. 
	\begin{enumerate}
		\item[1.] (L, L) is a NE iff $\max\{\rho,0\} \leq \beta \leq \min\{1 - \rho,1\}$.
		\item[2.] (L, H) is a NE iff $\max\{1- \rho,0\} \leq \beta \leq \min\{\dfrac{1}{\rho},1\}$.
		\item[3.] (H, L) is a NE iff $\max \{1- \dfrac{1}{\rho},0\} \leq \beta \leq \min\{\rho,1\}$.
		\item[4.] (H, H) is a NE iff $\max \{\dfrac{1}{\rho},0\} \leq \beta \leq \min\{1 - \dfrac{1}{\rho},1\}$.
	\end{enumerate}
\end{theorem}
The proof is in Appendix \ref{app:1}. 

Figure \ref{fig:equilibria_a} illustrates Theorem \ref{thm:eq1}. Observe that depending on the ratio of the lowest possible demand to the lowest possible supply $(\ch/\wl)$ we get the range of values of loyalty $(\beta)$ for which the various equilibria occur. Thus, Theorem \ref{thm:eq1} completely characterizes the Nash equilibrium for this case. It also shows that at least one pure strategy Nash equilibrium always exists for the game and gives the precise conditions for existence of multiple pure strategy equilibria for the game. 
When the customer loyalty for the two platforms is close to $1/2$, if the number of customers at the high price is significantly less than the number of workers at the low price, then (L, L) emerges as the Nash equilibrium. Consequently, the platforms operate at a low price which might be desirable to the customers. Conversely, if workers at the low price are fewer in number than the customers at the high price then (H, H) corresponds to the Nash equilibrium. 

\begin{figure}[h]
	\centering\includegraphics[scale=0.32]{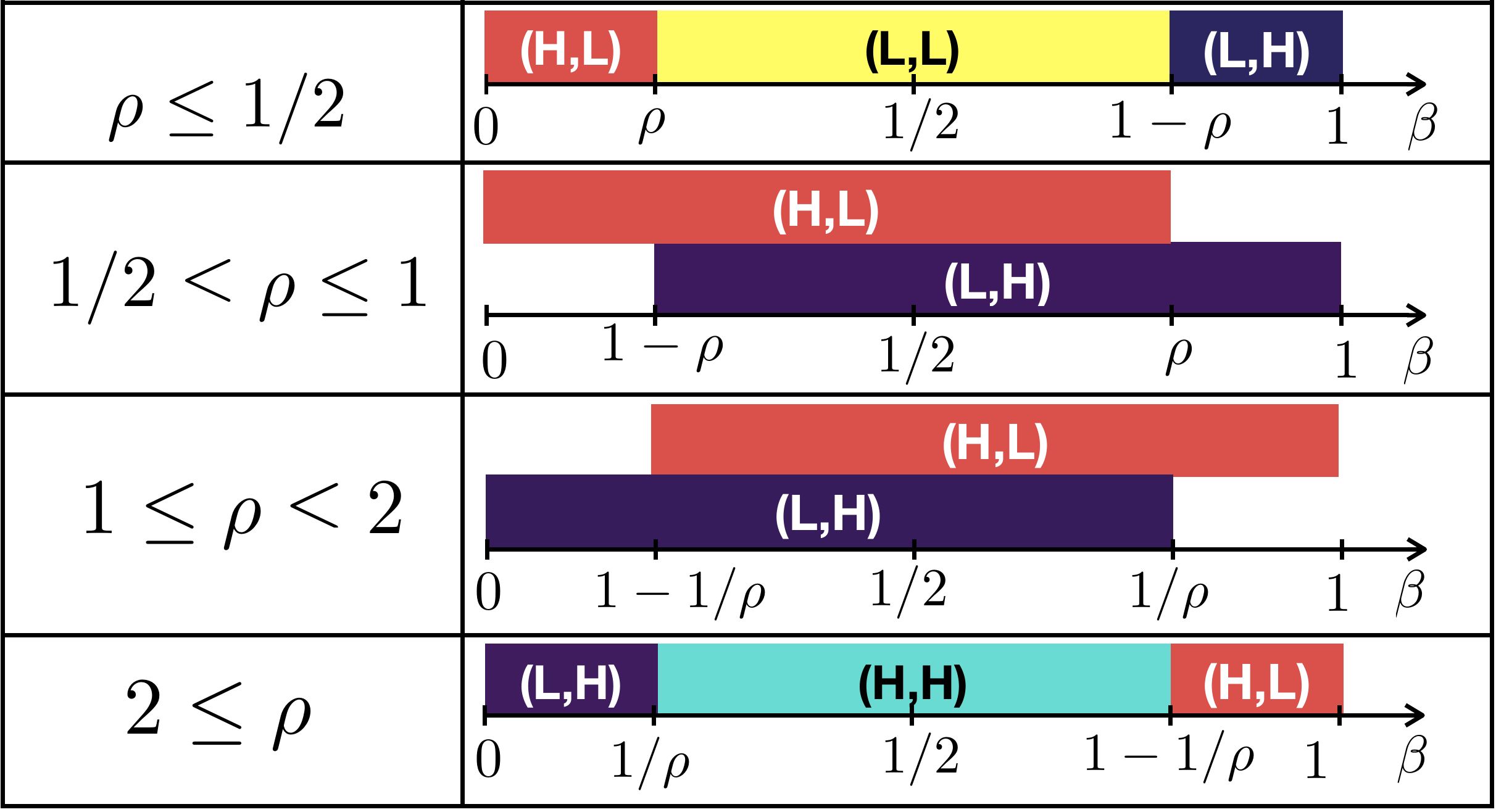}
	\caption{Dependence of pure strategy Nash equilibria for the throughput maximization game on customer loyalty $\beta$ and the ratio $\rho$ i.e., $\ch/\wl $.} \label{fig:equilibria_a}
\end{figure}

\subsection{$\pl < \ph \leq \pb$}
In this case, the market is supply limited at both $\pl$ and $\ph$. Observe that
\begin{align}
N_1^{\rm LL} &= \min\{\beta \cl,\beta \wl \}= \beta \wl, \label{eq:cll}\\
N_1^{\rm HH} &= \min\{\beta \ch,\beta \wh \}=\beta \wh,\label{eq:chh}\\
N_1^{\rm LH} &= \min \{\dfrac{\beta}{1-\beta}\ch, \beta \cl, \wl, \beta \wh \}\nonumber\\&=\min \{\wl, \beta \wh \},\label{eq:clh}\\
N_1^{\rm HL} &= \min \{\ch, \beta \cl, \dfrac{\beta}{1-\beta}\wl, \beta \wh \}\nonumber\\&=\min \{\dfrac{\beta}{1-\beta}\wl, \beta \wh \}\label{eq:chl}.
\end{align}
From Assumption \ref{ass:demandsuppply} and $0<\beta<1$, it follows that 
\begin{align}
&\beta \wh \leq \beta \ch < \dfrac{\beta}{1-\beta}\ch, \label{eq:ineqc1} \\
&\beta \wh \leq \beta \ch < \beta \cl, \label{eq:ineqc2} \\
&\beta \wh < \wh \leq \ch \label{eq:ineqc3}.
\end{align}
Here, Equation (\ref{eq:clh}) follows by combining Inequalities (\ref{eq:ineqc1}) and (\ref{eq:ineqc2}). Moreover, combining Inequalities (\ref{eq:ineqc2}) and (\ref{eq:ineqc3}) gives Equation (\ref{eq:chl}) 
Table \ref{tab:payoff_w} gives the payoff for Platform 1. 
\begin{table}[h]
	\begin{center}
		\setlength{\extrarowheight}{2pt}
		\begin{tabular}{cc|c|c|}
			& \multicolumn{1}{c}{} & \multicolumn{2}{c}{Platform $2$}\\
			& \multicolumn{1}{c}{} & \multicolumn{1}{c}{L} & \multicolumn{1}{c}{H} \\\cline{3-4}
			\multirow{2}*{Platform $1$} & L & $\beta \wl$ & $\min\{\wl, \beta \wh\}$ \\\cline{3-4}
			& H & $\min \{\dfrac{\beta}{1-\beta} \wl, \beta \wh \}$ & $\beta \wh$ \\\cline{3-4}
		\end{tabular}
		\caption {Payoff matrix for Platform 1 if $\pl < \ph \leq \pb$. Platform 1 is the row player, Platform 2 is the column player. The payoff for Platform 2 is same as the payoff of Platform 1 scaled by factor $ \frac{1-\beta}{\beta}$.}\label{tab:payoff_w}
	\end{center}
\end{table}
\begin{theorem}\label{thm:eq3}
	When $\pl < \ph \leq \pb$, (H, H) is always a Nash equilibrium and (L, L) is never an equilibrium. Moreover,
	\begin{enumerate}
		\item[1.] (L, H) is a NE iff $0 < \beta \leq \min\left \{1,\dfrac{\wl}{\wh}\right \}$.
		\item[2.] (H, L) is a NE iff $\max\left \{1-\dfrac{\wl}{\wh},0\right \} \leq \beta < 1$.
	\end{enumerate}
\end{theorem}
The proof is in Appendix \ref{app:2}. 

\subsection{$\pb \leq \pl< \ph$}
This case occurs when both platform announce prices higher than $ \pb $ and witness a demand shortage. Observe that, 
\begin{align}
N_1^{\rm LL} &= \min\{\beta \cl,\beta \wl \}=\beta \cl, \label{eq:bll}\\
N_1^{\rm HH} &= \min\{\beta \ch,\beta \wh \}=\beta \ch,\label{eq:bhh}.
\end{align}
From Assumption \ref{ass:demandsuppply} and $0<\beta<1$, it follows that 
\begin{align}
&\beta \cl < \cl \leq \wl, \label{eq:ineqb1} \\
&\beta \cl \leq \beta \wl < \beta \wh, \label{eq:ineqb2} \\
&\beta \cl \leq \beta \wl < \dfrac{\beta}{1-\beta} \wl \label{eq:ineqb3}.
\end{align} 
Thus, $N_1^{\rm LH}$ and $N_1^{\rm HL}$ can be rewritten as
\begin{align}
N_1^{\rm LH} &= \min \{\dfrac{\beta}{1-\beta}\ch, \beta \cl, \wl, \beta \wh \}\nonumber\\&=\min \{\dfrac{\beta}{1-\beta}\ch, \beta \cl\},\label{eq:blh}\\
N_1^{\rm HL} &= \min \{\ch, \beta \cl, \dfrac{\beta}{1-\beta}\wl, \beta \wh \}\nonumber\\&=\min \{\ch, \beta \cl\} \label{eq:bhl}.
\end{align}
The payoff matrix for Platform 1 is given by Table \ref{tab:payoff_c}. 
\begin{table}[h]
	\begin{center}
		\setlength{\extrarowheight}{2pt}
		\begin{tabular}{cc|c|c|}
			& \multicolumn{1}{c}{} & \multicolumn{2}{c}{Platform $2$}\\
			& \multicolumn{1}{c}{} & \multicolumn{1}{c}{L} & \multicolumn{1}{c}{H} \\\cline{3-4}
			\multirow{2}*{Platform $1$} & L & $\beta \cl$ & $\min \{\dfrac{\beta}{1-\beta}\ch, \beta \cl \}$ \\\cline{3-4} 
			& H & $\min \{\ch, \beta \cl\}$ & $\beta \ch$ \\\cline{3-4}
		\end{tabular}
		\caption {Payoff matrix for Platform 1 if $\pb \leq \pl< \ph$. The payoff for Platform 2 is same as the payoff of Platform 1 scaled by factor $ \frac{1-\beta}{\beta}$.}\label{tab:payoff_c}
	\end{center}
\end{table}
\begin{theorem}\label{thm:eq2}
	When $\ \pb \leq \pl < \ph$, (L, L) is always a Nash equilibrium and (H, H) is never an equilibrium. Moreover,\begin{enumerate}
		\item[1.] (L, H) is a NE iff $\max\{1-\dfrac{\ch}{\cl},0\} \leq \beta < 1$.
		\item[2.] (H, L) is a NE iff $0 < \beta \leq \min\{1,\dfrac{\ch}{\cl}\}$.
	\end{enumerate}
\end{theorem}
The proof is in Appendix \ref{app:3}.

Figures \ref{fig:equilibria_2}(a) and \ref{fig:equilibria_2}(b) respectively illustrate Theorems \ref{thm:eq2} and \ref{thm:eq3}. When $\pl=\pb$, the payoffs for the case when $\pb \leq \pl< \ph$ become a limiting case for the payoffs for the $\pl < \pb< \ph$ case. Likewise, when $\pb=\ph$, the payoffs for $\pl < \ph \leq \pb$ can be viewed as a limiting case of $\pl < \pb< \ph$. An interesting consequence of this theorem is the creation of new pure strategy equilibria. This can be visualised by comparing Figures \ref{fig:equilibria_a} and \ref{fig:equilibria_2}.
\begin{figure}[h]
	\centering\includegraphics[scale=0.35]{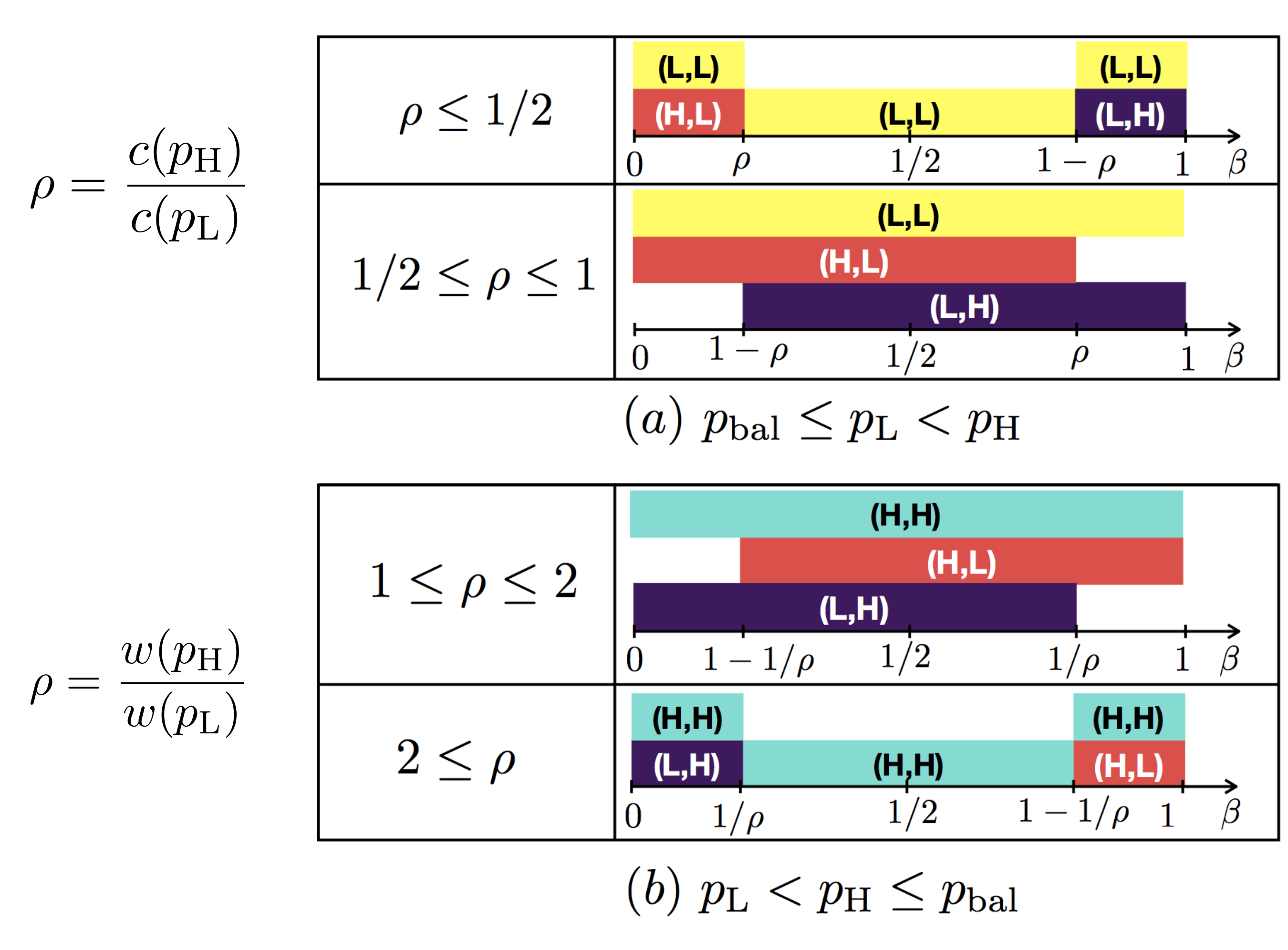}
	\caption{Dependence of Nash equilibria for the throughput maximization game on customer loyalty $\beta$ for {(a) $\pb \leq \pl < \ph$ and (b) $\pl < \ph \leq \pb$.}} \label{fig:equilibria_2}
\end{figure}

\section{Equilibrium strategies for revenue competition }\label{sec:eq_revenue} 
In this section, we analyse the revenue competition game between the two platforms. Recall that in each transaction at Platform $i$, a customer pays an amount $p_i$ and the platform retains a fraction $(1-\gamma)p_i$. Thus the revenue of Platform $i$ when Platform 1 plays $k$ and Platform 2 plays $\ell$ is given by $(1-\g)p_i N_i^{k\ell}$, where $i\in \{1,2\}$ and $k,\ell \in \{ \rm L,H\}.$ It is important to note that the revenue not only depends on the throughput but also the price per transaction. A high (resp., low) price fetches more workers (resp., customers) but deters customers (resp., workers). We observe that in addition to $\rho$, i.e., the fraction $\ch/ \wl$, the revenue intimately depends on the fractions defined below. 
\begin{align}
f &:=\dfrac{\pl}{\ph},\\
\rw&:=\dfrac{\wl}{\wh},\\
\rc&:=\dfrac{\ch}{\cl}.
\end{align}
Using Assumption \ref{ass:demandsuppply}, we have $0<f, \rw, \rc<1$ and $0<\rho$. 

\begin{remark}\label{remark:hh}
If (H, H) is a N.E. for the throughput competition game, then it must be a N.E. for the revenue competition game. Conversely, if (L, L) is a N.E. for the revenue competition game, then it must be a N.E. for the throughput competition game. Moreover, if $H$ is a dominant strategy for the throughput game, then $H$ is a dominant strategy for the revenue game. If $L$ is a dominant strategy for the revenue game then $L$ is also a dominant strategy for the throughput game.
\end{remark}
If (H, H) is a N.E. for the throughput competition game, then Platform 1 does not have an incentive to deviate and thus, $N_1^{\rm HH}\geq N_1^{\rm LH}$ which implies $\ph N_1^{\rm HH} \geq \pl N_1^{\rm LH}$. Likewise, $\ph N_2^{\rm HH} \geq \pl N_2^{\rm HL}$ and thus in this case, (H, H) is also a N.E. for the revenue competition game.  Along similar lines, it can be argued that if (L, L) is a N.E. for the revenue competition game, then it is also an equilibrium for the throughput competition game. If (L, L) is a N.E. for the revenue game then for $i=1,2$, $\pl N_i^{\rm LL} \geq \ph N_i^{\rm HL}$ which gives $N_i^{\rm LL} \geq \ph/\pl N_i^{\rm HL} > N_i^{\rm HL}$.
\par  In order to determine the equilibrium strategies for the revenue competition, we separately analyse the following three cases depending on the ordering of $\ph$ and $\pl$ with respect to $\pb$. 
\begin{enumerate}
\item[1.] $\pl < \ph \leq \pb,$
\item[2.] $\pb \leq \pl< \ph,$
\item[3.] $\pl < \pb< \ph.$
\end{enumerate}
For the remainder of this section, we refer to the revenue generated by a platform as the platform's payoff.

\subsection{$\pl < \ph \leq \pb$}
This case corresponds to a supply limited market in which both platforms announce prices lower than the market clearing price $\pb$. The entries of the payoff matrix for revenue for Platform $i$ can be obtained by multiplying the entries of the payoff matrix for the throughput game in Table \ref{tab:payoff_w} with a fraction $(1-\gamma)$ of the price that Platform $i$ charges to the customers, where $i=1,2$. For instance, if Platform 1 and Platform 2 respectively announce prices $\pl$ and $\ph$, then the throughput at Platform 1 is $\min \{ \wl, \beta \wh\}$. Since, each platform retains a fraction $(1-\gamma)$ of the price paid by the customer and passes on the rest to the worker, the revenue for Platform 1 is $(1-\gamma)\pl \min \{ \wl, \beta \wh\}$. Recall that from Theorem ~\ref{thm:routing}, the throughput at Platform 2 is $\frac{1-\beta}{\beta}$ times the throughput at Platform 1. Therefore when Platform 1 and Platform 2 play (L, H), then the throughput of Platform 2 is $\frac{1-\beta}{\beta}\min \{ \wl, \beta \wh\}$ and the revenue at Platform 2 is $\frac{(1-\beta)(1-\gamma) \ph}{\beta}\min \{ \wl, \beta \wh\}$. The resulting payoff matrix for Platform 1 (resp., Platform 2) is given respectively by Table \ref{tab:rev_payoff1_w} (resp., Table \ref{tab:rev_payoff2_w}). For the purpose of equilibrium analysis, all the entries in Table~\ref{tab:rev_payoff1_w} (resp., Table \ref{tab:rev_payoff2_w}) corresponds to the revenue of Platform 1 (resp., Platform 2) scaled by a factor $\frac{1}{1-\g}$ (resp., $\frac{(1-\beta)(1-\gamma)}{\beta}$). Since $\frac{1}{1-\g}>0$ (resp., $\frac{(1-\beta)(1-\gamma)}{\beta}$), the equilibria are unaltered under this monotonic transformation and these factors have been dropped for simplifying the payoffs without altering the equilibrium strategies.
 \begin{table}[h]
  \begin{center}
    \setlength{\extrarowheight}{2pt}
    \begin{tabular}{cc|c|c|}
      & \multicolumn{1}{c}{} & \multicolumn{2}{c}{Platform $2$}\\
      & \multicolumn{1}{c}{} & \multicolumn{1}{c}{L}  & \multicolumn{1}{c}{H} \\\cline{3-4}
      \multirow{2}*{Platform $1$}  & L & $\pl \beta \wl$ & $\pl \min\{\wl, \beta \wh\}$ \\\cline{3-4}
      & H & $\ph \min \{\dfrac{\beta}{1-\beta} \wl, \beta \wh \}$ & $\ph \beta \wh$ \\\cline{3-4}
    \end{tabular}
    \caption {Payoff matrix for Platform 1 for the case $\pl < \ph \leq \pb$ scaled by a factor $\frac{1}{1-\g}$. }\label{tab:rev_payoff1_w}
    \end{center}
  \end{table}
 \begin{table}[h]
  \begin{center}
    \setlength{\extrarowheight}{2pt}
    \begin{tabular}{cc|c|c|}
      & \multicolumn{1}{c}{} & \multicolumn{2}{c}{Platform $2$}\\
      & \multicolumn{1}{c}{} & \multicolumn{1}{c}{L}  & \multicolumn{1}{c}{H} \\\cline{3-4}
      \multirow{2}*{Platform $1$}  & L & $\pl \beta \wl$ & $\ph \min\{\wl, \beta \wh\}$ \\\cline{3-4}
      & H & $\pl \min \{\dfrac{\beta}{1-\beta} \wl, \beta \wh \}$ & $\ph \beta \wh$ \\\cline{3-4}
    \end{tabular}
    \caption {Payoff matrix for Platform 2 for the case $\pb \leq \pl< \ph$ scaled by a factor $\frac{\beta}{(1-\g) (1-\beta)}$. }    \label{tab:rev_payoff2_w}
    \end{center}
  \end{table} 
\par For a platform operating in a supply limited market, announcing a higher price to workers will increase throughput while fetching more revenue per transaction. From Remark~\ref{remark:hh} and Theorem~\ref{thm:eq3}, we expect (H, H) to be a N.E. for this case. It turns out that (H, H) is the only N.E. for this case.
\begin{theorem}\label{thm:eqw}
When $\pl < \ph \leq \pb$, (H, H) is the unique Nash equilibrium for the revenue competition game (see Table~\ref{tab:NEwrev}).
\end{theorem}
    \begin{table}[h]
 {\color{black}
  \begin{center}
    \setlength{\extrarowheight}{2pt}
    \begin{tabular}{cc|c|c|}
      & \multicolumn{1}{c}{} & \multicolumn{2}{c}{}\\
      & \multicolumn{1}{c}{} & \multicolumn{1}{c}{L}  & \multicolumn{1}{c}{H} \\\cline{3-4}
      \multirow{2}*{}  & L & 
      Never & 
      Never \\\cline{3-4} 
      & H & 
      Never & 
      Always \\\cline{3-4}
    \end{tabular}
    \caption {Conditions for existence of N.E. for revenue game when $\pl <\ph \leq \pb$. Platform 1 is the row player and Platform 2 is the column player. }    \label{tab:NEwrev}
    \end{center}}
  \end{table}
  Proof is in Appendix \ref{app:4}.
  


\subsection{$\pb \leq \pl< \ph$}
This case corresponds to a high price regime in which the market is demand limited. As before, the revenue of each platform is obtained by scaling the throughput given in Table \ref{tab:payoff_c} by the price announced by the platform and a factor $\frac{1}{(1-\gamma)}$ (resp., $\frac{(1-\beta)(1-\gamma)}{\beta}$) for Platform 1 (resp., Platform 2). The payoff matrix for Platform 1 is given by Table \ref{tab:rev_payoff1_c} and that of Platform 2 is given by Table \ref{tab:rev_payoff2_c}. Recall that $\rho_{\rm c}$ and $f$  respectively denote $\frac{\ch}{\cl}$ and $\frac{\pl}{\ph}$.
 \ Observe that $\frac{\rho_{\rm c}}{f}$ corresponds to the ratio of revenue at high price to that at low price in the presence of a single platform in the market.
\begin{table}[h]
  \begin{center}
    \setlength{\extrarowheight}{2pt}
    \begin{tabular}{cc|c|c|}
      & \multicolumn{1}{c}{} & \multicolumn{2}{c}{Platform $2$}\\
      & \multicolumn{1}{c}{} & \multicolumn{1}{c}{L}  & \multicolumn{1}{c}{H} \\\cline{3-4}
      \multirow{2}*{Platform $1$}  & L & $ \pl \beta \cl$ & $ \pl \min \{\dfrac{\beta}{1-\beta}\ch, \beta \cl \}$ \\\cline{3-4} 
      & H & $ \ph \min \{\ch, \beta \cl\}$ & $ \ph \beta \ch$ \\\cline{3-4}
    \end{tabular}
    \caption {Payoff matrix for Platform 1 for the case $\pb \leq \pl< \ph$ scaled by a factor $\dfrac{1}{1-\g}$. }\label{tab:rev_payoff1_c}
    \end{center}
  \end{table}
  
 \begin{table}[h]
  \begin{center}
    \setlength{\extrarowheight}{2pt}
    \begin{tabular}{cc|c|c|}
      & \multicolumn{1}{c}{} & \multicolumn{2}{c}{Platform $2$}\\
      & \multicolumn{1}{c}{} & \multicolumn{1}{c}{L}  & \multicolumn{1}{c}{H} \\\cline{3-4}
      \multirow{2}*{Platform $1$}  & L & $ \pl \beta \cl$ & $ \ph \min \{\dfrac{\beta}{1-\beta}\ch, \beta \cl \}$ \\\cline{3-4} 
      & H & $ \pl \min \{\ch, \beta \cl\}$ & $ \ph \beta \ch$ \\\cline{3-4}
    \end{tabular}
    \caption {Payoff matrix for Platform 2 for the case $\pb \leq \pl< \ph$ scaled by a factor $\dfrac{\beta}{(1-\g) (1-\beta)}$. }    \label{tab:rev_payoff2_c}
    \end{center}
  \end{table}
We now write the conditions under which different equilibria arise for the revenue competition game.
\\ \noindent(i) (H, H) is a N.E. iff
\begin{align}
& \ph \beta \ch \geq \pl \min \{ \dfrac{\beta}{1-\beta}\ch,\beta \cl\},\\
& \ph \beta \ch \geq \pl \min \{ \ch,\beta \cl \}
\end{align}
Substituting, $\beta, \rho_{\rm c}$ and $f$,
\begin{align}
&  \beta \rho_{\rm c} \geq f \min \{ \dfrac{\beta}{1-\beta}\rho_{\rm c},\beta \},\\
&  \beta \rho_{\rm c} \geq f \min \{ \rho_{\rm c},\beta \}
\end{align}
Dividing by $\beta \rho_{\rm c}$, we get that (H, H) is a N.E. iff
\begin{align}
&   \dfrac{1}{f} \geq  \min \{ \dfrac{1}{1-\beta},\dfrac{1}{\rho_{\rm c}} \},\\
&   \dfrac{1}{f} \geq  \min \{ \dfrac{1}{\beta},\dfrac{1}{\rho_{\rm c}} \}
\end{align}
\\ \noindent(ii) (L, L) is a N.E. iff
\begin{align}
& \pl \beta \cl \geq \ph \min \{ \ch,\beta \cl \},\\
& \pl \beta \cl \geq \ph \min \{ \dfrac{\beta}{1-\beta}\ch,\beta \cl\}
\end{align}
Substituting $\rho_{\rm c}, \beta$ and $f$,
\begin{align}
& f \beta  \geq  \min \{ \rho_{\rm c},\beta  \},\\
& f \beta  \geq  \min \{ \dfrac{\beta}{1-\beta}\rho_{\rm c},\beta \}
\end{align}
Dividing by $\beta$,  we see that (L, L) is a N.E. iff
\begin{align}
& f   \geq  \min \{  \dfrac{\rho_{\rm c}}{\beta}, 1 \},\\
& f   \geq  \min \{ \dfrac{\rho_{\rm c}}{1-\beta},1 \}
\end{align}
\noindent(iii) (L, H) is a N.E. iff
\begin{align}
& \pl \min\{ \dfrac{\beta}{1-\beta}\ch,\beta\cl \} \geq \ph \beta \ch, \\
& \ph \min\{ \dfrac{\beta}{1-\beta}\ch,\beta\cl \} \geq \pl \beta \cl.
\end{align}
Substituting $\rho_{\rm c}, \beta, f$,
\begin{align}
& f \min\{ \dfrac{\beta}{1-\beta}\rho_{\rm c},\beta \} \geq  \beta \rho_{\rm c}, \label{eq:revct3} \\
&  \min\{ \dfrac{\beta}{1-\beta}\rho_{\rm c},\beta \} \geq f \beta \label{eq:revct4}.
\end{align}
Rewriting the above inequalities after multiplying (\ref{eq:revct3}) (resp., (\ref{eq:revct4})) with a factor $\dfrac{1}{f \beta \rho_{\rm c}}$ (resp., $\dfrac{1}{\beta}$), 
\begin{align}
&  \min\{ \dfrac{1}{1-\beta},\dfrac{1}{\rho_{\rm c}} \} \geq   \dfrac{1}{f},  \label{eq:revct1}\\
&  \min\{ \dfrac{\rho_{\rm c}}{1-\beta},1 \} \geq f  \label{eq:revct2} .
\end{align}
\noindent(iv) 
Since the two platforms only differ in their respective customer loyalties, the conditions under which (H, L) is a N.E. can be obtained from the conditions for (L, H) to be a N.E. by replacing the factor $1-\beta$ with $\beta$ in equations (\ref{eq:revct1}) and  (\ref{eq:revct2}). Consequently, (H, L) is a N.E. iff
\begin{align}
&  \min\{ \dfrac{1}{\beta},\dfrac{1}{\rho_{\rm c}} \} \geq   \dfrac{1}{f}, \\
&  \min\{ \dfrac{\rho_{\rm c}}{\beta},1 \} \geq f  .
\end{align}

Table~\ref{tab:NEc} summarizes the conditions for the existence of N.E. for $\pb \leq \pl< \ph$ where $f$ lies in $(0,1)$. Here, the entry under row $k$ and column $\ell$ gives the necessary and sufficient conditions for existence of N.E. for the revenue competition game at $(k, \ell)$.
 \begin{table}[h]
 {\color{black}
  \begin{center}
    \setlength{\extrarowheight}{2pt}
    \begin{tabular}{cc|c|c|}
      & \multicolumn{1}{c}{} & \multicolumn{2}{c}{}\\
      & \multicolumn{1}{c}{} & \multicolumn{1}{c}{L}  & \multicolumn{1}{c}{H} \\\cline{3-4}
      \multirow{2}*{}  & L & $f   \geq  \min \{  \dfrac{\rho_{\rm c}}{\beta}, 1 \},\min \{ \dfrac{\rho_{\rm c}}{1-\beta},1 \}$ & $ \dfrac{1}{f} \leq \min\{ \dfrac{1}{1-\beta},\dfrac{1}{\rho_{\rm c}} \}, f \leq \min\{ \dfrac{\rho_{\rm c}}{1-\beta},1 \}  $ \\\cline{3-4} 
      & H & $ \dfrac{1}{f} \leq \min\{ \dfrac{1}{\beta},\dfrac{1}{\rho_{\rm c}} \}, f \leq \min\{ \dfrac{\rho_{\rm c}}{\beta},1 \}  $ & $ \dfrac{1}{f} \geq  \min \{ \dfrac{1}{1-\beta},\dfrac{1}{\rho_{\rm c}} \},  \min \{ \dfrac{1}{\beta},\dfrac{1}{\rho_{\rm c}} \}$ \\\cline{3-4}
    \end{tabular}
    \caption {Conditions for existence of N.E. for revenue competition game when $\pb \leq \pl< \ph$. Platform 1 is the row player and Platform 2 is the column player. Here, the entry under row $k$ and column $\ell$ gives the necessary and sufficient conditions for existence of N.E. at $(k, \ell)$.}    \label{tab:NEc}
    \end{center}}
  \end{table}

Next, we show at least one of the above conditions (as listed in Table~\ref{tab:NEc}) hold for any value of $ \beta $. 
\begin{theorem} \label{thm:revpbplph} 
A pure strategy N.E. always exists for the case $\pb \leq \pl< \ph$. 
\end{theorem}
The proof is in Appendix \ref{app:5}.

\subsection{$\pl < \pb< \ph$}
In this case the market is supply limited at the low price and demand limited at the high price. As done previously for the demand limited and supply limited cases, we can obtain the payoffs for the revenue game, using the throughput of the platforms from equations (\ref{eq:nll}), (\ref{eq:nlh}), (\ref{eq:nhl}) and (\ref{eq:nhh}). The resulting payoff matrix for Platform 1 is given by Table \ref{tab:rev_payoff1_o} and that of Platform 2 is given by Table \ref{tab:rev_payoff2_o}. 
 \begin{table}[h]
  \begin{center}
    \begin{tabular}{cc|c|c|}
      & \multicolumn{1}{c}{} & \multicolumn{2}{c}{Platform $2$}\\
      & \multicolumn{1}{c}{} & \multicolumn{1}{c}{L}  & \multicolumn{1}{c}{H} \\\cline{3-4}
      \multirow{2}*{Platform $1$}  & L  & $\pl \beta \wl$ & \shortstack{$\pl \min \{\frac{\beta}{1-\beta}\ch, \beta \cl,$\\ $\qquad \quad \wl, \beta \wh \}$} \\\cline{3-4}
      & H & \shortstack{$\ph \min \{c(\ph), \beta c(\pl),$\\ $\qquad \quad  \frac{\beta}{1-\beta}w(\gamma \pl),  \beta w( \gamma \ph) \}$} & $\ph \beta \ch$ \\\cline{3-4}
    \end{tabular}
    \caption {Payoff matrix for Platform 1 for the case $\pl < \pb < \ph$ scaled by a factor $\frac{1}{1-\g}$. Platform 1 is the row player and Platform 2 is the column player.}\label{tab:rev_payoff1_o}
    \end{center}
  \end{table}

   \begin{table}[h]
  \begin{center}
    \setlength{\extrarowheight}{2pt}
    \begin{tabular}{cc|c|c|}
      & \multicolumn{1}{c}{} & \multicolumn{2}{c}{Platform $2$}\\
      & \multicolumn{1}{c}{} & \multicolumn{1}{c}{L}  & \multicolumn{1}{c}{H} \\\cline{3-4}
      \multirow{2}*{Platform $1$}  & L & $\pl \beta \wl$ & \shortstack{$\ph \min \{\frac{\beta}{1-\beta}\ch, \beta \cl,$\\ $\qquad \quad  \wl, \beta \wh \}$} \\\cline{3-4}
      & H & \shortstack{$\pl \min \{c(\ph), \beta c(\pl),$\\ $\qquad \quad  \frac{\beta}{1-\beta}w(\gamma \pl),  \beta w( \gamma \ph) \}$} & $\ph \beta \ch$ \\\cline{3-4}
    \end{tabular}
    \caption {Payoff matrix for Platform 2 for the case $\pl < \pb < \ph$ scaled by a factor $\frac{\beta}{(1-\g)(1-\beta)}$. Platform 1 is the row player and Platform 2 is the column player.}\label{tab:rev_payoff2_o}
    \end{center}
  \end{table}

 \noindent(i) (L, L) is a N.E. iff
\begin{align}
& \pl \beta \wl \geq \ph \min \{ \ch,\beta \cl, \dfrac{\beta}{1-\beta}\wl, \beta \wh \},\\
& \pl \beta \wl \geq \ph \min \{ \dfrac{\beta}{1-\beta}\ch,\beta \cl, \wl, \beta \wh\}.
\end{align}
Dividing by $\beta$ and substituting $f$,
\begin{align}
& f  \wl \geq  \min \{ \dfrac{\ch}{\beta}, \cl, \dfrac{\wl}{1-\beta},  \wh \}, \label{eq:ro_ineqll1}\\
& f  \wl \geq  \min \{ \dfrac{\ch}{1-\beta}, \cl, \dfrac{\wl}{\beta},  \wh\} \label{eq:ro_ineqll2}.
\end{align}
Using Assumption~\ref{ass:demandsuppply} and $0<\beta<1$, we have 
\begin{align}
\wl < \min\{ \cl,\wh, \dfrac{\wl}{1-\beta}, \dfrac{\wl}{\beta}\}
\end{align}
and therefore (\ref{eq:ro_ineqll1}) and (\ref{eq:ro_ineqll2}) hold iff 
\begin{align}
&\wl f \geq \dfrac{\ch}{\beta}, \\
&\wl f \geq \dfrac{\ch}{1-\beta}.
\end{align}
Therefore substituting $\rho$ above, we get that (L, L) is  a N.E. iff
\begin{align}
& f \geq \dfrac{\rho}{\beta}, \\
& f \geq \dfrac{\rho}{1-\beta}.
\end{align}
 \noindent(ii) (H, H) is a N.E. iff
 \begin{align}
& \ph \beta \ch \geq \pl \min \{ \ch,\beta \cl, \dfrac{\beta}{1-\beta}\wl, \beta \wh \},\\
& \ph \beta \ch \geq \pl \min \{ \dfrac{\beta}{1-\beta}\ch,\beta \cl, \wl, \beta \wh\}.
\end{align}
Dividing by $\beta$ and substituting $f$,
\begin{align}
& \dfrac{1}{f}  \ch \geq  \min \{ \dfrac{\ch}{\beta}, \cl, \dfrac{\wl}{1-\beta},  \wh \}, \\
&\dfrac{1}{f}  \ch \geq  \min \{ \dfrac{\ch}{1-\beta}, \cl, \dfrac{\wl}{\beta},  \wh\} .
\end{align} 
 From Assumption~\ref{ass:demandsuppply}, observe that 
\begin{align}
\dfrac{1}{\rho \rw} = \dfrac{\wh}{\ch}> 1, \dfrac{\rho}{\rc} = \dfrac{\cl}{\wl} >1.
\end{align}
Now, substituting $\rho,\ \rc, \ \rw$,
\begin{align}
& \dfrac{1}{f}   \geq  \min \{ \dfrac{1}{\beta}, \dfrac{1}{\rc}, \dfrac{1}{\rho(1-\beta)},  \dfrac{1}{\rho \rw} \}, \label{eq:ro_ineqhh1}\\
& \dfrac{1}{f}   \geq  \min \{ \dfrac{1}{1-\beta}, \dfrac{1}{\rc}, \dfrac{1}{\rho(\beta)},  \dfrac{1}{\rho \rw} \}.\label{eq:ro_ineqhh2}
\end{align}
 \\ \noindent(iii) (L, H) is a N.E. iff
 \begin{align}
& \pl \min \{ \dfrac{\beta}{1-\beta}\ch,\beta \cl, \wl, \beta \wh\} \geq \ph \beta \ch ,\\
& \ph \min \{ \dfrac{\beta}{1-\beta}\ch,\beta \cl, \wl, \beta \wh\} \geq \pl \beta \wl.
\end{align}
Dividing by $\beta$ and substituting $f$,
 \begin{align}
&  \min \{ \dfrac{\ch}{1-\beta}, \cl, \dfrac{\wl}{\beta},  \wh\} \geq  \dfrac{1}{f} \ch ,\\
&  \min \{ \dfrac{\ch}{1-\beta}, \cl, \dfrac{\wl}{\beta},  \wh\} \geq  f  \wl. \label{eq:ro_ineqlh1_temp}
\end{align}
Note that in (\ref{eq:ro_ineqlh1_temp}), $\wl < \min\{\cl, \dfrac{\wl}{\beta},  \wh\} $. Now, for (\ref{eq:ro_ineqlh1_temp}) to hold true, $f \wl \leq \dfrac{\ch}{1-\beta} $.
Thus, (L, H) is a N.E. iff
 \begin{align}
&  \min \{ \dfrac{1}{1-\beta}, \dfrac{1}{\rc}, \dfrac{1}{\rho \beta}, \dfrac{1}{\rho \rw} \} \geq  \dfrac{1}{f}  ,\\
&   \dfrac{\rho}{1-\beta} \geq f .
\end{align}

 \noindent(iv) As in the previous case, the conditions under which (H, L) is a N.E. can be obtained from the conditions for (L, H) to be a N.E. by replacing the factor $1-\beta$ with $\beta$. Thus, (H, L) is a N.E. iff
 \begin{align}
&  \min \{ \dfrac{1}{\beta}, \dfrac{1}{\rc}, \dfrac{1}{\rho( 1- \beta)}, \dfrac{1}{\rho \rw} \} \geq  \dfrac{1}{f}  ,\\
&   \dfrac{\rho}{\beta} \geq f .
\end{align}

Table~\ref{tab:NEo} summarizes the conditions for the existence of various pure strategy N.E for this case.

\begin{table}[h]
 {\color{black}
  \begin{center}
    \setlength{\extrarowheight}{2pt}
\begin{tabular}{cc|c|c|}
      & \multicolumn{1}{c}{} & \multicolumn{2}{c}{}\\
      & \multicolumn{1}{c}{} & \multicolumn{1}{c}{L}  & \multicolumn{1}{c}{H} \\\cline{3-4}
      \multirow{2}*{}  &L  & 
      $f \geq \frac{\rho}{1-\beta},\frac{\rho}{\beta}$ 
      & \shortstack{$\frac{1}{f}\leq \min \{ \frac{1}{1-\beta}, \frac{1}{\rc}, \frac{1}{\rho \beta}, \frac{1}{\rho \rw} \},$\\ $ f \leq  \frac{\rho}{1-\beta}$} \\\cline{3-4} 
      & H &
       \shortstack{$\frac{1}{f}\leq \min \{ \frac{1}{\beta}, \frac{1}{\rc}, \frac{1}{\rho (1-\beta)}, \frac{1}{\rho \rw} \},$\\ $ f \leq  \frac{\rho}{\beta}$} & 
       \shortstack{$ \frac{1}{f}   \geq   \min \{ \frac{1}{1-\beta}, \frac{1}{\rc}, \frac{1}{\rho\beta},  \frac{1}{\rho \rw} \},$\\ $\frac{1}{f} \geq \min \{ \frac{1}{\beta}, \frac{1}{\rc}, \frac{1}{\rho(1-\beta)},  \frac{1}{\rho \rw} \} $ }\\\cline{3-4}
    \end{tabular}
    \caption {Conditions for existence of N.E. for revenue game when $\pl <\pb<  \ph$. Here, $\rho=\frac{\ch}{\wl}, \rw=\frac{\wl}{\wh}, \rc=\frac{\ch}{\cl}$. Platform 1 is the row player and Platform 2 is the column player.}    \label{tab:NEo}
    \end{center}}
  \end{table}
    \begin{theorem}
  A pure strategy N.E. always exists for the case $ \pl< \pb < \ph$. \label{thm:revplpbph} 
  \end{theorem}
Proof is in Appendix \ref{app:6}.

\section{Discussion}\label{sec:discussion}
We now present a discussion on the results obtained for the throughput and revenue competition games under different market conditions and loyalty functions.
\subsection{Supply limited market}
\begin{enumerate}
	\item For throughput game,
		\begin{enumerate}
			\item Both platforms together achieve their best throughput at (H, H) (Table~\ref{tab:payoff_tot}). Playing high price is a weakly dominant strategy for the platforms and thus (H, H) is always a N.E.
				 \begin{table}[h]
			 \begin{center}
 			\setlength{\extrarowheight}{2pt}
 			\begin{tabular}{cc|c|c|}
 				& \multicolumn{1}{c}{} & \multicolumn{2}{c}{Platform $2$}\\
 				& \multicolumn{1}{c}{} & \multicolumn{1}{c}{L} & \multicolumn{1}{c}{H} \\\cline{3-4}
 				\multirow{2}*{Platform $1$} & L & $ \wl$ & $\min\{\dfrac{\wl}{\beta},\wh\}$ \\\cline{3-4}
 				& H & $\min\{\dfrac{\wl}{1-\beta},\wh\}$& $ \wh$ \\\cline{3-4}
			\end{tabular}
		 \caption {Total number of workers served for supply limited market $(\pl < \ph \leq \pb)$ for fixed $\beta$.}\label{tab:payoff_tot}
 			\end{center}
			\end{table}		
			\item In a market which is supply limited at $(p_1, p_2)$, the following three constraints on $w_1$ need to be met --
				\begin{align}
					 w_1& \leq w(\gamma p_1), \label{eq:w11}\\
					w_1& \leq \dfrac{\beta}{1-\beta} w(\gamma p_2), \label{eq:w12}\\
					w_1& \leq \beta w(\max\{\gamma p_1, \gamma p_2\})\label{eq:w1t}.
				\end{align}
			(\ref{eq:w11}) corresponds to the supply constraint at Platform 1 $w_1 \leq w(\gamma p_1) $. Recall that $w_2^*=\dfrac{1-\beta}{\beta}w_1^*$ and thus (\ref{eq:w11}) corresponds to the supply constraint at Platform 2. (\ref{eq:w1t}) captures the upper-bound on the total supply available in the market.
			
			\item (L, H) is a N.E. iff $ \wh \leq \wl / \beta$. This occurs when (\ref{eq:w1t}) provides a tighter upper bound than (\ref{eq:w11}), i.e., the total number of workers served is $\wh$. 
			Platform 2 is strictly better off playing $H$ than $L$. Platform 1 is indifferent to either of the two strategies when $\beta \wh \leq \wl$. If there's a perturbation in the market which results in $\beta \wh$ exceeding $\wl$ then Platform 1 is no longer indifferent and strictly prefers $H$ over $L$. (H, L) is a N.E. iff $ \wh \leq \wl / (1-\beta)$. If $\wh$ is significantly greater than $\wl$, then (H, H) is the only equilibrium.
			
			\item When $\beta=1/2$, 
				\begin{enumerate}
				\item (L, H) and (H, L) are N.E. iff $\wl \leq \wh \leq 2\wl$.				
				\item If there was only a single throughput maximizing platform in the market, it would have played the high price, paying an amount $\gamma \ph$ to each of the $\wh$ workers availing service. 
				\par In contrast, (L, H) and (H, L) may also be equilibria in the case of two platforms. When, (L, H) is an equilibrium then of the total $\wh$ workers availing service, $\beta \wh$ get served at Platform 1 receiving an amount $\gamma \pl$ while the rest get served at Platform 2 at $\gamma \ph$. Observe that in the case of a single as well as two platforms, a total of $\wh$ workers get served but a fraction $\beta \wh$ receive a lower payment in the case of two platforms. Consequently, the aggregate utility of workers can be lower when two identical platforms operate in a supply limited market.
				\end{enumerate}
				
			\item When $\beta<1/2$, Platform 1 has lower customer loyalty than Platform 2.
				\begin{enumerate}
				\item If (H, L) is a N.E. then (L, H) is also a N.E but the converse is not true. This follows from noting that in the payoff matrix $\dfrac{\beta}{1-\beta}\wl < \wl$ when $\beta <1/2$.
				\end{enumerate}
		\end{enumerate}

	\item For the revenue competition game,
		\begin{enumerate}
			\item Playing high price is a strictly dominant strategy for both platforms and thus (H, H) is always a N.E. Moreover, (H, H) is the only N.E. for the revenue game.

		\end{enumerate}		

\end{enumerate}


\subsection{Demand limited market}
\begin{enumerate}
	\item For the throughput competition game,
		\begin{enumerate}
			\item $L$ is a dominant strategy for both Platforms. The Platforms achieve their highest throughput at (L, L) and it is always a N.E.
			\item (L, H) is an equilibrium iff $\cl \leq \dfrac{\ch}{1-\beta}$ and (H, L) is an equilibrium ff $\cl \leq \dfrac{\ch}{\beta}$.
		\end{enumerate}
	The remaining analysis for throughput game in a demand limited setting is analogous to the supply limited market with strategies $H$ and $L$ interchanged.
	
	\item For the revenue competition game,
		\begin{enumerate}
			\item Let $\rh/\rl$ denote the ratio $\ph \ch/\pl \cl$. Table~\ref{tab:NEcinterpret} enlists the conditions for existence of N.E (Table~\ref{tab:NEc} ) in terms of $\rh/\rl$ and $\pl/\ph$.
			\item Note that (H, H) is never an equilibrium for the throughput game. However, observe that is $\rh/\rl \geq 1$, then (H, H) is a N.E. At (H, H) Platform 1 generates $(1-\gamma)\beta \ph \ch$ and Platform 2 generates $(1-\gamma)(1-\beta) \ph \ch$.
			\item If $\cl \leq \ch/1-\beta, \ch/\beta$, then (L, L) ceases to be an equilibrium.
		\end{enumerate}
	 \begin{table}[h]
{\color{black}
 \begin{center}
  \setlength{\extrarowheight}{2pt}
  \begin{tabular}{cc|c|c|}
   & \multicolumn{1}{c}{} & \multicolumn{2}{c}{}\\
   & \multicolumn{1}{c}{} & \multicolumn{1}{c}{L} & \multicolumn{1}{c}{H} \\\cline{3-4}
   \multirow{2}*{} & L & 
   $\rh/\rl \leq \beta, 1-\beta$ & 
   $ \ph/\pl \leq 1/1-\beta  , \quad 1-\beta \leq \rh/\rl \leq 1 $ 
   \\\cline{3-4} 
   & H &  $ \ph/\pl \leq 1/\beta  , \quad \beta \leq \rh/\rl \leq 1 $ & 
$\rh/\rl \geq 1$ or $ \ph/\pl \geq 1/1-\beta, 1/\beta$ \\\cline{3-4}
  \end{tabular}
  \caption {Conditions for existence of N.E. for revenue competition game when $\pb \leq \pl< \ph$ in terms of $\rh/\rl$ and $\pl/\ph$.}  \label{tab:NEcinterpret}
  \end{center}}
 \end{table}

\end{enumerate}
{\color{purple}

\subsection{Market with price dependent customer loyalty}
Here, we analyse the more general case where $\beta$ is allowed to be a function of $(p_1, p_2)$ such that a larger fraction of customers avails service at the cheaper platform when the platforms price their services differently. Specifically,			
\begin{equation}\label{eq:def}
\beta(p_1,p_2)=
\begin{cases}
&1/2 \text{ if } p_1=p_2, \\
&1/2+t \text{ if } p_1<p_2, \text{ where } 0<t<1/2, \\
&1/2-t \text{ if } p_1>p_2.
\end{cases}
\end{equation}
				
	 \begin{lemma}\label{lemma:4}
 Consider the two-player simultaneous move game with symmetric payoffs given by Table (\ref{tab:payoffexsymm}). A pure strategy N.E. always exists for this game.
 \end{lemma}
 \begin{proof}
We prove this by finding at least one N.E. for all the cases below,
\begin{enumerate}
\item If $a_{11} \geq a_{21}$ then (L, L) is a N.E.
\item If $a_{22} \geq a_{12}$ then (H, H) is a N.E.
\item If $a_{11}< a_{21}$ and $a_{22} < a_{12}$ then both (L, H) and (H, L) are equilibria.
\end{enumerate}
Thus, we have argued for the existence of N.E. for all values of payoffs and this completes the proof.
\end{proof}
\begin{table}[h]
{
 \begin{center}
  \setlength{\extrarowheight}{2pt}
  \begin{tabular}{cc|c|c|}
   & \multicolumn{1}{c}{} & \multicolumn{2}{c}{}\\
   & \multicolumn{1}{c}{} & \multicolumn{1}{c}{$B_1$} & \multicolumn{1}{c}{$B_2$} \\\cline{3-4}
   \multirow{2}*{} & $A_1$ & 
   $(a_{11}, a_{11})$ & 
   $(a_{12}, a_{21})$ \\\cline{3-4} 
   &$ A_2$& 
   $(a_{21}, a_{21})$& 
   $(a_{22}, a_{22})$ \\\cline{3-4}
  \end{tabular}
  \caption {Payoffs for a two player simultaneous move game. Here, $a_{11}$ denotes Platform 1's payoff when Platform1 plays $A_1$ and Platform 2 plays $B_1$.}  \label{tab:payoffexsymm}
  \end{center}}
 \end{table}

\begin{theorem}
In the presence of price dependent customer loyalty, let $ t\in [0,1/2] $ and 
\begin{equation}
\beta(p_1,p_2)=
\begin{cases}
&1/2 \text{ if } p_1=p_2, \\
&1/2+t \text{ if } p_1<p_2, \\
&1/2-t \text{ if } p_1>p_2,
\end{cases}
\end{equation}
a pure strategy N.E. always exists for the throughput and revenue competition games.
\end{theorem}
\begin{proof}
From the perspectives of users, all that the platforms differ in is the prices they charge and since the prices announced completely determines the loyalty, the observed throughputs for the platforms should be symmetric with respect to the prices charged by them, i.e., $N_1^{k\ell}=N_2^{\ell k}$, where $k,\ell \in {L, H}$. Further, since both the platforms only have the prices $\pl,\ph$ as available strategies, the revenue generated by them is also symmetric with respect to prices. It follows that the payoff of Platform 1 at (L, H) is same as the payoff of Platform 2 at (H, L) and vice versa. Also, when both platforms announce same prices, their payoffs are identical. Thus, the game is of the form described in Table~\ref{tab:payoffexsymm} and using Lemma~\ref{lemma:4}, it follows that a pure strategy N.E. always exists. 
\end{proof}
\par Recall that the solution to the user allocation problem (S) was computed at a fixed price. Using Theorem~\ref{thm:routing}, we analyse the supply and demand limited markets in the presence of price dependent loyalty.}
		\begin{enumerate}
			\item If $\beta$ is a function $(p_1, p_2)$ as defined above, the throughput for the supply limited market is given in Tables~\ref{tab:payoff_w1beta_var},~\ref{tab:payoff_w2beta_var} and \ref{tab:payoff_tot_beta_var}. (H, H) is a N.E. if $\wh/2\geq\wl$. Unlike the case of fixed loyalty $\beta$, when Platform 2 is playing H, Platform 1 can achieve better throughput by playing L if $\wh/2<\wl$. Also note that for Platform 1, (L, H) yields a better throughput than (H, L). This is because the total number of workers who are served is $\min\{\dfrac{\wl} {1/2+t},\wh\}$ for both (L, H) and (H, L) but Platform 1's share is larger when it plays the lower price. 
			 
			 \begin{table}[h]
			 \begin{center}
  			\setlength{\extrarowheight}{2pt}
  			\begin{tabular}{cc|c|c|}
   				& \multicolumn{1}{c}{} & \multicolumn{2}{c}{}\\
   				& \multicolumn{1}{c}{} & \multicolumn{1}{c}{L} & \multicolumn{1}{c}{H} \\\cline{3-4}
   				\multirow{2}*{} & L & $\wl/2 $ & $\min\{\wl, (1/2+t)\wh\}$ \\\cline{3-4}
   				& H & $\min \{\dfrac{1/2-t}{1/2+t} \wl, (1/2-t) \wh \}$ & $\wh/2 $ \\\cline{3-4}
  			\end{tabular}
 			 \caption {Throughput for Platform 1 for supply limited market $(\pl < \ph \leq \pb)$ with $\beta$ as a function of $(p_1, p_2)$. Platform 1 is the row player, Platform 2 is the column player.}\label{tab:payoff_w1beta_var}
  			\end{center}
 			\end{table}
			
			 \begin{table}[h]
			 \begin{center}
  			\setlength{\extrarowheight}{2pt}
  			\begin{tabular}{cc|c|c|}
   				& \multicolumn{1}{c}{} & \multicolumn{2}{c}{}\\
   				& \multicolumn{1}{c}{} & \multicolumn{1}{c}{L} & \multicolumn{1}{c}{H} \\\cline{3-4}
   				\multirow{2}*{} & L & $\wl/2$ & $\min \{\dfrac{1/2-t}{1/2+t} \wl, (1/2-t) \wh \}$ \\\cline{3-4}
   				& H & $\min\{\wl, (1/2+t)\wh\}$ & $\wh/2$ \\\cline{3-4}
  			\end{tabular}
 			 \caption {Throughput for Platform 2 for supply limited market $(\pl < \ph \leq \pb)$ with $\beta$ as a function of $(p_1, p_2)$. Platform 1 is the row player, Platform 2 is the column player.}\label{tab:payoff_w2beta_var}
  			\end{center}
 			\end{table}
			
			 \begin{table}[h]
			 \begin{center}
  			\setlength{\extrarowheight}{2pt}
  			\begin{tabular}{cc|c|c|}
   				& \multicolumn{1}{c}{} & \multicolumn{2}{c}{Platform $2$}\\
   				& \multicolumn{1}{c}{} & \multicolumn{1}{c}{L} & \multicolumn{1}{c}{H} \\\cline{3-4}
   				\multirow{2}*{Platform $1$} & L & $ \wl$ & $\min\{\dfrac{\wl} {1/2+t},\wh\}$ \\\cline{3-4}
   				& H & $\min\{\dfrac{\wl} {1/2+t},\wh\}$ & $ \wh$ \\\cline{3-4}
  			\end{tabular}
 			 \caption {Total number of workers served when $\pl < \ph \leq \pb$ and $\beta$ is a function of price $(p_1, p_2)$.}\label{tab:payoff_tot_beta_var}
  			\end{center}
 			\end{table}
	\item If $\beta$ is allowed to be a function of $(p_1, p_2)$ as defined in (\ref{eq:def}), then for the demand-limited market, (L, L) is always the unique N.E. for throughput game. The throughput of Platform 1 is given in Table~\ref{tab:payoff_c_betavar}. Observe that $N_1^{\rm HL}= \min \{\ch, (1/2-t) \cl \}\leq(1/2-t) \cl < 1/2) \cl =N_1^{\rm HL}$. Moreover, $\ch/2 < \cl/2 < (1/2+t)\cl/2$ and $\ch/2 <\ch <\dfrac{1/2+t}{1/2-t}\ch$ and thus $N_1^{\rm HH}\leq N_1^{\rm LH}$. Therefore, $L$ emerges as a strictly dominant strategy for the platforms in this case. 

			 \begin{table}[h]
			 \begin{center}
  			\setlength{\extrarowheight}{2pt}
  			\begin{tabular}{cc|c|c|}
   				& \multicolumn{1}{c}{} & \multicolumn{2}{c}{Platform $2$}\\
   				& \multicolumn{1}{c}{} & \multicolumn{1}{c}{L} & \multicolumn{1}{c}{H} \\\cline{3-4}
   				\multirow{2}*{Platform $1$} & L & $\cl/2 $ & $\min\{ \dfrac{1/2+t}{1/2-t}\ch, (1/2+t)\cl \}$ \\\cline{3-4}
   				& H & $\min \{\ch, (1/2-t) \cl \}$ & $\ch/2 $ \\\cline{3-4}
  			\end{tabular}
 			\caption {Throughput for Platform 1 for demand limited market $(\pb \leq \pl < \ph)$ with $\beta$ as a function of $(p_1, p_2)$.}\label{tab:payoff_c_betavar}
  			\end{center}
 			\end{table}

	\end{enumerate}
{\color{purple}
We now analyse the case where all the customers go to the platform with the lower price. This is essentially a market in the absence of loyalty where customers simply affiliate with the platform which charges them less. More precisely, $\beta$ is described as follows,			
				\begin{equation}\label{eq:def2}
					 \beta(p_1,p_2)=
								 \begin{cases}
 					                &1/2 \text{ if } p_1=p_2, \\
 					                &1 \text{ if } p_1<p_2,\\
					                &0 \text{ if } p_1>p_2.
								 \end{cases}
				\end{equation}
Observe that the customers choose the platform with the lower price independent of all other characteristics of the platforms and therefore therefore it can be argued that $(L, H)$ is an equilibrium iff $(H, L)$ is an equilibrium. 
\begin{enumerate}
\item 		For the supply limited market (Table~\ref{tab:payoff_w1beta_var_strict}), 
\begin{enumerate} \item For throughput game, 
(L, L) is always a N.E. and (H, H) is a N.E. iff $\wh \geq 2 \wl$ and (L, H), (H, L) are never equilibria. .
			 \begin{table}[h]
			 \begin{center}
  			\setlength{\extrarowheight}{2pt}
  			\begin{tabular}{cc|c|c|}
   				& \multicolumn{1}{c}{} & \multicolumn{2}{c}{Platform $2$}\\
   				& \multicolumn{1}{c}{} & \multicolumn{1}{c}{L} & \multicolumn{1}{c}{H} \\\cline{3-4}
   				\multirow{2}*{Platform $1$} & L & $\wl/2 $ & $\wl$ \\\cline{3-4}
   				& H & $0$ & $\wh/2 $ \\\cline{3-4}
  			\end{tabular}
 			 \caption {Throughput for Platform 1 for supply limited market $(\pl < \ph \leq \pb)$ with $\beta$ as described in (\ref{eq:def2}).}\label{tab:payoff_w1beta_var_strict}
  			\end{center}
 			\end{table}
			\item For revenue game, (L, L) is always a N.E. and (H, H) is a N.E. iff $\ph \wh \geq 2 \pl \wl$ and (L, H), (H, L) are never equilibria.
\end{enumerate}
\item 		For the demand limited market (Table~\ref{tab:payoff_c_betavar_strict}),
\begin{enumerate}
\item For throughput game, (L, L) is the only equilibrium and it always exists.%
\item For revenue game, (L, L) is always a N.E. and (H, H) is a N.E. iff $\ph \ch \geq 2 \pl \cl$ and (L, H), (H, L) are never equilibria.
\end{enumerate}
			 \begin{table}[h]
			 \begin{center}
  			\setlength{\extrarowheight}{2pt}
  			\begin{tabular}{cc|c|c|}
   				& \multicolumn{1}{c}{} & \multicolumn{2}{c}{Platform $2$}\\
   				& \multicolumn{1}{c}{} & \multicolumn{1}{c}{L} & \multicolumn{1}{c}{H} \\\cline{3-4}
   				\multirow{2}*{Platform $1$} & L & $\cl/2 $ & $\cl $ \\\cline{3-4}
   				& H & $0$ & $\ch/2 $ \\\cline{3-4}
  			\end{tabular}
 			\caption {Throughput for Platform 1 for demand limited market $(\pb \leq \pl < \ph)$ with $\beta$ as described in (\ref{eq:def2}).}\label{tab:payoff_c_betavar_strict}
  			\end{center}
 			\end{table}
\item 		When $ \pl<\pb < \ph$ (Table~\ref{tab:payoff_o_betavar_strict}), 
\begin{enumerate}
\item For throughput game, (L, L) is always a N.E. and (H, H) is a N.E. iff $\ch \geq 2 \wl$%
\item For revenue game, (L, L) is always a N.E. and (H, H) is a N.E. iff $\ph \ch \geq 2 \pl \wl$ and (L, H), (H, L) are never equilibria.
\end{enumerate}
			\begin{table}[h]
			 \begin{center}
  			\setlength{\extrarowheight}{2pt}
  			\begin{tabular}{cc|c|c|}
   				& \multicolumn{1}{c}{} & \multicolumn{2}{c}{Platform $2$}\\
   				& \multicolumn{1}{c}{} & \multicolumn{1}{c}{L} & \multicolumn{1}{c}{H} \\\cline{3-4}
   				\multirow{2}*{Platform $1$} & L & $\wl/2 $ & $\wl $ \\\cline{3-4}
   				& H & $0$ & $\ch/2 $ \\\cline{3-4}
  			\end{tabular}
 			\caption {Throughput for Platform 1 for market with $ \pl<\pb < \ph$ and $\beta$ as described in (\ref{eq:def2}).}\label{tab:payoff_o_betavar_strict}
  			\end{center}
 			\end{table}
Thus, in the presence of strictly price dependent loyalties as described by (\ref{eq:def2}), the platforms are forced to play indentical prices at equilibrium. 
\end{enumerate}
}

\section{Summary}\label{sec:summary}

We modelled duopolistic platform competition in two-sided markets in which two platforms provide identical services to the users at potentially different prices. We characterized the Nash equilibrium strategies of a price competition game for throughput and revenue between two platforms operating in this setting. We assumed that customers have a preference or loyalty to the platforms, {and the resulting allocation maximized the aggregate utility of the users.} We found that for each value of the loyalty, there exists a pure strategy Nash equilibrium and characterized it. We also established the existence of a Nash equilibrium in the case of price dependent loyalty constraints in which more customers avail service at the platform which announces the lower price. Further questions to be explored include the consideration of  allocations based on individual utility maximization, asymmetric commissions and loss of efficiency due to the presence of intermediaries.

\section*{\large Appendix}
\appendix
\normalsize
\section{Proof of Theorem \ref{thm:eq1}} \label{app:1}
\begin{proof} 
We separately find necessary and sufficient conditions for all four cases.
\\(i) (L, L) is an equilibrium iff the following hold, 
\begin{align}
\beta \wl &\geq  \min \{\dfrac{\beta}{1-\beta}\ch, \beta \cl, \wl, \beta \wh \} ,\label{eq:AL1}\\
\beta \wl &\geq  \min \{\ch, \beta \cl, \dfrac{\beta}{1-\beta}\wl,  \beta \wh \}
\label{eq:AL2}.
\end{align}

Here, (\ref{eq:AL1}) ensures that Platform 2 does not have an incentive to deviate and likewise (\ref{eq:AL2}) ensures that Platform 1 also does not deviate unilaterally. Note that the following inequalities hold,
{
\begin{align}
\beta \wl  &< \beta \cl, \label{eq:ineq1}\\ 
\beta \wl  &< \wl,\label{eq:ineq2}\\
\beta \wl  &< \beta \wh,\label{eq:ineq3}\\
\beta \wl &< \dfrac{\beta}{1-\beta}\wl\label{eq:ineq4}.
\end{align}}

Observe that (\ref{eq:ineq1}), (\ref{eq:ineq3}) follow from Assumption \ref{ass:demandsuppply} and (\ref{eq:ineq2}), (\ref{eq:ineq4}) are a trivial consequence of loyalty $(\beta)$ lying in $(0,1)$. Using (\ref{eq:ineq1}), (\ref{eq:ineq2}), (\ref{eq:ineq3}), it follows that condition (\ref{eq:AL1}) holds iff $\dfrac{\beta}{1-\beta} \ch \leq \beta \wl$.
Similarly from (\ref{eq:ineq1}), (\ref{eq:ineq3}), (\ref{eq:ineq4}) it follows that (\ref{eq:AL2}) holds iff $\ch \leq \beta \wl$. Thus, (L, L) is an equilibrium iff $\dfrac{\beta}{1-\beta} \ch \leq \beta \wl$ and $\ch \leq \beta \wl$ where $\beta \in (0,1)$.
\\(ii) (L, H) is an equilibrium iff the following hold, 

\begin{align}
 \min  \{\dfrac{\beta}{1-\beta}\ch, \beta \cl, \wl, \beta \wh \} \geq \beta \ch,\label{eq:ALH1}\\
\min \{\dfrac{\beta}{1-\beta}\ch, \beta \cl, \wl, \beta \wh \} \geq \beta \wl
\label{eq:ALH2}.
\end{align}

Assumption \ref{ass:demandsuppply} and $0<\beta<1$ also give the following,

\begin{align}
\beta \ch &< \dfrac{\beta}{1-\beta}\ch, \label{eq:ineq5}\\ 
\beta \ch &< \beta \cl,\label{eq:ineq6}\\
\beta \ch &< \beta \wh.\label{eq:ineq7}
\end{align}

It is easy to see that $\beta \ch \leq \wl$ together with (\ref{eq:ineq5}), (\ref{eq:ineq6}), (\ref{eq:ineq7}) is necessary and sufficient for (\ref{eq:ALH1}) to hold. Similarly, (\ref{eq:ineq2}), (\ref{eq:ineq3}), (\ref{eq:ineq1}) and $\beta \wl \leq \dfrac{\beta}{1-\beta} \ch$ ensures that (\ref{eq:ALH2}) holds. Thus, (L, H) is a NE iff $\beta \ch \leq \wl$ and $\beta \wl \leq \dfrac{\beta}{1-\beta} \ch$.
\\(iii) (H, L) is an equilibrium iff the following hold, 

\begin{align}
\min \{\ch, \beta \cl, \dfrac{\beta}{1-\beta}\wl,  \beta \wh \} &\geq   \beta \wl,\label{eq:AHL1} \\
\min \{\ch, \beta \cl, \dfrac{\beta}{1-\beta}\wl,  \beta \wh \} &\geq  \beta \ch. \label{eq:AHL2}
\end{align}

The proof for this case is analogous to Case (ii). 
\\(iv) (H, H) is an equilibrium iff the following hold, 

\begin{align}
\beta \ch &\geq  \min \{\dfrac{\beta}{1-\beta}\ch, \beta \cl, \wl, \beta \wh \} ,\label{eq:AH1}\\
\beta \ch &\geq  \min \{\ch, \beta \cl, \dfrac{\beta}{1-\beta}\wl,  \beta \wh \}\label{eq:AH2}.
\end{align}

Observe that (\ref{eq:ineq5}), (\ref{eq:ineq6}), (\ref{eq:ineq7}) along with $\wl \leq \beta \ch$ is necessary and sufficient for (\ref{eq:AH1}) to hold. We have $\beta \ch < \ch$. Combining this with (\ref{eq:ineq6}), (\ref{eq:ineq7}), $\dfrac{\beta}{1-\beta} \wl \leq \beta \ch$ is a necessary and sufficient for (\ref{eq:AH2}) to hold. Hence, we get a NE at (H, H) iff $\wl \leq \beta \ch$ and $\dfrac{\beta}{1-\beta} \wl \leq \beta \ch$. \end{proof}

\section{Proof of Theorem \ref{thm:eq3}} \label{app:2}

  \begin{lemma}\label{lemma:lem_throughput}
  Consider a general two-player simultaneous move game in two strategies $S_1, S_2$ available to each player and payoffs given by Table (\ref{tab:lemmathroughput}) such that $0<\beta<1$ and $a_1 \leq a_2$. There always exists a pure strategy N.E. for this game.
  Further, 
  \begin{enumerate}
  \item[(i)]$(S_1, S_1)$ can never be a N.E.
  \item[(ii)]$(S_2, S_2)$ is always a N.E.
  \item[(iii)]$(S_1, S_2)$ s a N.E. iff $0<\beta \leq \min \{ 1, \frac{a_1}{a_2} \}$.
  \item[(iv)]$(S_2, S_1)$ is a N.E. iff $1- \frac{a_1}{a_2} \leq \beta$.
  \end{enumerate}
  \end{lemma}
     \begin{table}[h]
 {
  \begin{center}
    \setlength{\extrarowheight}{2pt}
    \begin{tabular}{cc|c|c|}
      & \multicolumn{1}{c}{} & \multicolumn{2}{c}{}\\
      & \multicolumn{1}{c}{} & \multicolumn{1}{c}{$S_1$}  & \multicolumn{1}{c}{$S_2$} \\\cline{3-4}
      \multirow{2}*{}  & $S_1$ & 
      $\beta a_1$ & 
      $\min \{ a_1, \beta a_2\}$ \\\cline{3-4} 
      & $S_2$ & 
      $\min \{ \frac{\beta}{1-\beta} a_1, \beta a_2\}$& 
      $\beta a_2$ \\\cline{3-4}
    \end{tabular}
    \caption {Payoffs for a two player simultaneous move game with $0<\beta<1$ and $a_1 \leq a_2$. }    \label{tab:lemmathroughput}
    \end{center}}
  \end{table}
 \begin{proof}[of Lemma \ref{lemma:lem_throughput}]
  \begin{enumerate}
  \item[(i)] $(S_1, S_1)$ is a N.E. iff $\beta a_1 \geq \min \{a_1, \beta a_2 \}$ and $\beta a_1 \geq \min \{\frac{\beta}{1-\beta}a_1,  \beta a_2 \}$. However, with $0<\beta<1$ we have
\begin{align}
&\beta a_1 < a_1, \label{eq:ineqc5} \\
&\beta a_1 < \beta a_2, \label{eq:ineqc6} \\
&\beta a_1 < \dfrac{\beta}{1-\beta} a_2. \label{eq:ineqc7} 
\end{align}
From Inequalities (\ref{eq:ineqc5}), (\ref{eq:ineqc5}) and (\ref{eq:ineqc7}) it follows that $(S_1, S_1)$ can never be a N.E.
\item[(ii)] $(S_2,S_2)$ is a N.E. iff $\beta a_2 \geq \min \{a_1, \beta a_2 \}$ and $\beta a_2 \geq \min \{\frac{\beta}{1-\beta}a_1, \beta a_2 \}$. Observe that given any $a,b \in \mathbb{R}$, the inequality $a\geq \min\{a,b\}$ holds by definition and consequently $(S_2,S_2)$ is always a N.E..
\item[(iii)] $(S_1,S_2)$  is a N.E. iff
\begin{align}
&  \min \{a_1, \beta a_2 \} \geq \beta a_1 ,\label{eq:ineqc8}\\
& \min \{a_1, \beta a_2 \} \geq \beta a_2 \label{eq:ineqc9} .
\end{align}
Inequality (\ref{eq:ineqc8}) directly follows when $0<\beta<1$. Note that (\ref{eq:ineqc9}) holds iff $\beta a_2 \leq a_1$. Thus, $(S_1,S_2)$ is a Nash Equilibrium iff $0 < \beta \leq \min\{1,\frac{a_1}{a_2}\}$.\\
\item[(iv)]$(S_2,S_1)$ is a N.E. iff
\begin{align}
&  \min \{\dfrac{\beta}{1-\beta}a_1,  \beta a_2 \} \geq \beta a_1 ,\label{eq:ineqc10}\\
& \min \{\dfrac{\beta}{1-\beta}a_1,  \beta a_2 \} \geq \beta a_2 \label{eq:ineqc11} .
\end{align}
As above, using $0<\beta<1$,  (\ref{eq:ineqc10}) always hold. However, (\ref{eq:ineqc11}) is true iff $\beta a_2 \leq \dfrac{\beta}{1-\beta}a_1$ or equivalently $\beta \geq 1-\frac{a_1}{a_2}$ with $\beta \in (0,1)$.
\end{enumerate}
 \end{proof}
\begin{proof} [of Theorem \ref{thm:eq3}]
Theorem~\ref{thm:eq3} is a consequence of Lemma~\ref{lemma:lem_throughput} with $S_1=L$, $S_2=H$, $a_1=\wl$ and $a_2=\wh$ and noting that from Assumption \ref{ass:demandsuppply}, $\wl \leq \wh$.
\\
\end{proof}


\section{Proof of Theorem~\ref{thm:eq2}} \label{app:3}
\begin{proof}
Theorem~\ref{thm:eq3} follows from Lemma~\ref{lemma:lem_throughput} with $S_1=H$, $S_2=L$, $a_1=\ch$, $a_2=\cl$ and by replacing $\beta$ by $1-\beta$. Further, upon multiplying the payoff matrix with a factor $\frac{\beta}{1-\beta}$ and observing that from Assumption \ref{ass:demandsuppply}, $\ch \leq \cl$, Theorem~\ref{thm:eq2} emerges as a consequence of Lemma~\ref{lemma:lem_throughput}.
\end{proof}

\section{Proof of Theorem~\ref{thm:eqw}} \label{app:4}
\begin{proof}
\noindent(i) (H, H) is a N.E. iff
\begin{align}
& \ph \beta \wh \geq \pl \min\{\wl,\beta\wh\}, \\
& \ph \beta \wh \geq \pl \min\{\dfrac{\beta}{1-\beta}\wl,\beta \wh\}.
\end{align}
Substituting $\rho_{\rm w}=\dfrac{\wl}{\wh}$ and $f=\dfrac{\pl}{\ph}$, the above inequalities can be rewritten as,
\begin{align}
&  \beta  \geq f \min\{\rho_{\rm w},\beta\}, \label{eq:rw_ineqhh1temp}\\
&  \beta  \geq f \min\{\dfrac{\beta}{1-\beta}\rho_{\rm w},\beta \}\label{eq:rw_ineqhh2temp}.
\end{align}
For $a,b,k \in \mathbb{R}^+$, we have $\min \{ka,kb\} =k \min \{a,b\}$ and thus we can divide both sides of (\ref{eq:rw_ineqhh1temp}) and (\ref{eq:rw_ineqhh2temp}) by $\beta$, to get,
\begin{align}
&  1  \geq  f \min \{\dfrac{\rho_{\rm w}}{\beta},1\}, \label{eq:rw_ineqhh1},\\
&  1  \geq f \min\{\dfrac{\rho_{\rm w}}{1-\beta},1\}\label{eq:rw_ineqhh2}.
\end{align}
Noting that $\min \{\dfrac{\rho_{\rm w}}{\beta},1\} \leq 1$ and combining this with $0<f<1$, we see that (\ref{eq:rw_ineqhh1}) always holds for this case. A similar argument suffices to establish the truth of (\ref{eq:rw_ineqhh1}). Thus, (H, H) is a N.E.
 \noindent(ii) (L, L) is a N.E. iff
\begin{align}
& \pl \beta \wl \geq \ph \min\{\dfrac{\beta}{1-\beta}\wl,\beta \wh\}, \\
& \pl \beta \wl \geq \ph \min\{\wl,\beta\wh\}.
\end{align}
The above inequalities can be expressed in terms of $\rho_{\rm w}$, $\beta$ and $f$ as,
\begin{align}
&  f\beta \rho_{\rm w}  \geq \min\{\dfrac{\beta}{1-\beta}\rho_{\rm w},\beta \},\\
&  f\beta \rho_{\rm w}  \geq  \min\{\rho_{\rm w},\beta\} .
\end{align}
Dividing both sides by $\beta \rho_{\rm w}$, we get,
{\color{black}
\begin{align}
&  f   \geq \min\{\dfrac{1}{1-\beta},\dfrac{1}{\rho_{\rm w}} \}, \label{eq:rw_ineqll1}\\
&  f   \geq  \min\{\dfrac{1}{\beta},\dfrac{1}{\rho_{\rm w}}\}  \label{eq:rw_ineqll2}.
\end{align}}
Recall that $0<f,\rho_{\rm w},\beta<1$ and thus $\dfrac{1}{1-\beta},\dfrac{1}{\beta},\dfrac{1}{\rho_{\rm w}}>1$. Therefore, (\ref{eq:rw_ineqll1}) and (\ref{eq:rw_ineqll2}) do not hold and (L, L) is not a N.E. for this case.
\noindent(iii) (L, H) is a N.E. iff
\begin{align}
& \pl \min\{\wl,\beta \wh\} \geq \ph \beta \wh, \\
& \ph \min\{\wl,\beta \wh\} \geq \pl \beta \wl.
\end{align}
Substituting $\rho_{\rm w},\beta$ and $f$ we get,
\begin{align}
& f \min\{\rho_{\rm w},\beta \} \geq \beta , \\
&  \min\{\rho_{\rm w},\beta \} \geq f \beta \rho_{\rm w}.
\end{align}
Dividing by $\beta$,
\begin{align}
&{ \color{black}f \min\{\dfrac{\rho_{\rm w}}{\beta},1 \} \geq 1} ,\label{eq:rw_ineqlh1} \\
&  \min\{\dfrac{1}{\rho_{\rm w}},\dfrac{1}{\beta} \} \geq f .
\end{align}
If (\ref{eq:rw_ineqlh1}) is true, then $f  \geq f\min\{\dfrac{\rho_{\rm w}}{\beta},1 \} \geq 1$ which contradicts the fact that $0<f<1$.
\noindent(iv) (H, L) is a N.E. iff
\begin{align}
&{ \color{black}f \min\{\dfrac{\rho_{\rm w}}{1-\beta},1 \} \geq 1} \label{eq:rw_ineqhl1}, \\
&  \min\{\dfrac{1}{\rho_{\rm w}},\dfrac{1}{1-\beta} \} \geq f .
\end{align}
As above, (\ref{eq:rw_ineqhl1}) does not hold and thus neither (L,H) nor (H, L) are equilibria.
\end{proof}

\section{Proof of Theorem~\ref{thm:revpbplph}} \label{app:5}
\begin{proof}
 We now analyse the conditions in Table~\ref{tab:NEc} in more detail by separately solving the three cases -- $ \beta <1/2,\beta=1/2 $ and $\beta>1/2$.

 \begin{enumerate}

 \item {\color{black}$\beta < \dfrac{1}{2}$}
 	\begin{enumerate}
 		\item $\rho_{\rm c} \leq \beta$

		Table~\ref{tab:NEc_beta_leqhalf_1} gives the conditions for the existence of various pure strategy equilibria.  We now separately analyse the two cases,
\begin{enumerate}
\item $f \geq \dfrac{\rho_{\rm c}}{\beta}$: (L, L) is a N.E.
\item When $f < \dfrac{\rho_{\rm c}}{\beta}$, further consider the two cases,
	\begin{enumerate}
	\item $f< \beta$: (H, H) is a N.E.
	\item $f\geq  \beta $: (H, L) is a N.E. since $\beta \leq f <\dfrac{\rho_{\rm c}}{\beta} $.
	\end{enumerate}
\end{enumerate}
Thus, there always exists a N.E. in pure strategies for this case.

 \begin{table}[h]
	{{\color{black}
		\begin{center}
			\setlength{\extrarowheight}{2pt}
			\begin{tabular}{cc|c|c|}
				& \multicolumn{1}{c}{} & \multicolumn{2}{c}{}\\
				& \multicolumn{1}{c}{} & \multicolumn{1}{c}{L}  & \multicolumn{1}{c}{H} \\\cline{3-4}
				\multirow{2}*{}  & L & $   \dfrac{\rho_{\rm c}}{\beta} \leq f  $ & $  {1-\beta} \leq f \leq  \dfrac{\rho_{\rm c}}{1-\beta} $ \\\cline{3-4} 
				& H & $ \beta \leq f \leq  \dfrac{\rho_{\rm c}}{\beta} $ & $ f \leq \beta $ \\\cline{3-4}
			\end{tabular}
			\caption {Conditions for existence of N.E. for revenue competition game when $\pb \leq \pl< \ph$ and $\rho_{\rm c} \leq \beta < 1-\beta$. }   \label{tab:NEc_beta_leqhalf_1}
	\end{center}}}
\end{table}
 \begin{table}[h]
	{{\color{black}
		\begin{center}
			\setlength{\extrarowheight}{2pt}
			\begin{tabular}{cc|c|c|}
				& \multicolumn{1}{c}{} & \multicolumn{2}{c}{}\\
				& \multicolumn{1}{c}{} & \multicolumn{1}{c}{L}  & \multicolumn{1}{c}{H} \\\cline{3-4}
				\multirow{2}*{}  & L & $   \dfrac{\rho_{\rm c}}{f} \leq \beta <1 $ & $  {1-f},1-\dfrac{\rho_{\rm c}}{f}  \leq \beta <1   $ \\\cline{3-4} 
				& H & $ 0<\beta \leq f, \dfrac{\rho_{\rm c}}{f} $ & $ f \leq \beta<1 $ \\\cline{3-4}
			\end{tabular}
			\caption {Conditions for existence of N.E. when $\pb \leq \pl< \ph$ and $\rho_{\rm c} \leq \beta \leq 1-\beta$ }   
	\end{center}}}
\end{table}
\item  $\beta < \rho_{\rm c} < 1-\beta$

Table~\ref{tab:NEc_beta_leqhalf_2} gives the conditions for the existence of N.E. Note that (L, L) is never a N.E. Moreover, at least one amongst (H, H), or (H, L) is always a N.E. Thus, a pure strategy N.E. always exists. Recall that the ratio $\dfrac{\rho_{\rm c}}{f}$ represents the ratio of the revenue at high price to that at a low price. Observe that for this case, $\dfrac{\rho_{\rm c}}{f} \geq 1$ indeed corresponds to the condition for existence of N.E. at (H, H).
 \begin{table}[h]
	{\color{black}
		\begin{center}
			\setlength{\extrarowheight}{2pt}
			\begin{tabular}{cc|c|c|}
				& \multicolumn{1}{c}{} & \multicolumn{2}{c}{}\\
				& \multicolumn{1}{c}{} & \multicolumn{1}{c}{L}  & \multicolumn{1}{c}{H} \\\cline{3-4}
				\multirow{2}*{}  & L & Never & $1-\beta \leq f \leq \dfrac{\rho_{\rm c}}{1-\beta}$ \\\cline{3-4} 
				& H & $  \rho_{\rm c} \leq f$ & $ f \leq \rho_{\rm c}$ \\\cline{3-4}
			\end{tabular}
			\caption {Conditions for existence of N.E. for revenue competition game when $\pb \leq \pl< \ph.$ and $\beta < \rho_{\rm c} < 1-\beta$.}  \label{tab:NEc_beta_leqhalf_2}
	\end{center}}
\end{table}
\item  $1-\beta \leq \rho_{\rm c}$

Table~\ref{tab:NEc_beta_leqhalf_3} gives the conditions for N.E. Note that (L, L) is never a N.E. and (L, H) is a N.E. iff (H, L) is a N.E. Moreover, at least one amongst (H, H), or, (L, H) and (H, L) is always a N.E. Thus, a pure strategy N.E. always exists.
 \begin{table}[h]
	{\color{black}
		\begin{center}
			\setlength{\extrarowheight}{2pt}
			\begin{tabular}{cc|c|c|}
				& \multicolumn{1}{c}{} & \multicolumn{2}{c}{}\\
				& \multicolumn{1}{c}{} & \multicolumn{1}{c}{L}  & \multicolumn{1}{c}{H} \\\cline{3-4}
				\multirow{2}*{}  & L & Never & $  \rho_{\rm c} \leq f$ \\\cline{3-4} 
				& H & $  \rho_{\rm c} \leq f$ & $ f \leq \rho_{\rm c}$ \\\cline{3-4}
			\end{tabular}
			\caption {Conditions for existence of N.E.  for revenue competition game when $\pb \leq \pl< \ph$ and $\beta <1-\beta \leq \rho_{\rm c}$.}   \label{tab:NEc_beta_leqhalf_3}
	\end{center}}
\end{table}
\end{enumerate}

\item {\color{black}$\beta = \dfrac{1}{2}$}

Table~\ref{tab:NEc_beta_half} gives the conditions for N.E. Observe that in this case, at least one pure strategy N.E. exists for all values of $0<f<1$. Further $(L, H)$ is a N.E. iff $(H, L)$ is a N.E.
\begin{table}[h]
	{{\color{black}
			\begin{center}
				\setlength{\extrarowheight}{2pt}
				\begin{tabular}{cc|c|c|}
					& \multicolumn{1}{c}{} & \multicolumn{2}{c}{}\\
					& \multicolumn{1}{c}{} & \multicolumn{1}{c}{L}  & \multicolumn{1}{c}{H} \\\cline{3-4}
					\multirow{2}*{}  & L & $   2\rho_{\rm c} \leq f  $ & $  \dfrac{1}{f} \leq \min\{2,\dfrac{1}{\rc}\}, f \leq 2 \rc $ \\\cline{3-4} 
					& H &$ \dfrac{1}{f} \leq \min\{2,\dfrac{1}{\rc}\}, f \leq 2 \rc $ & $ \dfrac{1}{f} \geq \min\{2,\dfrac{1}{\rc}\}$ \\\cline{3-4}
				\end{tabular} 			\caption {Conditions for existence of N.E. for revenue competition game when $\pb \leq \pl< \ph$ and $\beta=\dfrac{1}{2}$. }    \label{tab:NEc_beta_half}
	\end{center}}}
\end{table}
 	\item $\beta > \dfrac{1}{2}$.
	
	Note that the platforms only differ in their customer loyalties ($\beta$ and $1-\beta$). $\beta<1/2$ corresponds to a higher value of loyalty for Platform 2 whereas $\beta>1/2$  corresponds to higher loyalty for Platform 1. This case can be argued as done previously for the case $\beta < \dfrac{1}{2}$. 

 \end{enumerate}
 \end{proof}
 
 \section{Proof of Theorem~\ref{thm:revplpbph}} \label{app:6}
 
  \begin{lemma}\label{lemma:pure}
  Consider a general two-player simultaneous move game in two strategies with payoffs given by Table \ref{tab:payoffex}. If the payoffs are such that the inequalities $a_{22}<a_{12}$, $b_{22}>b_{21}$, $b_{11}>b_{12}$ and $a_{11}<a_{21}$ are never simultaneously satisfied. 
  Additionally, if $a_{22}>a_{12}$, $b_{22}<b_{21}$, $b_{11}<b_{12}$ and  $a_{11}>a_{21}$ cannot simultaneously hold then there always exists a pure strategy N.E. for this game.
  \end{lemma}
     \begin{table}[h]
 {
  \begin{center}
    \setlength{\extrarowheight}{2pt}
    \begin{tabular}{cc|c|c|}
      & \multicolumn{1}{c}{} & \multicolumn{2}{c}{}\\
      & \multicolumn{1}{c}{} & \multicolumn{1}{c}{L}  & \multicolumn{1}{c}{H} \\\cline{3-4}
      \multirow{2}*{}  & L & 
      $(a_{11}, b_{11})$ & 
      $(a_{12}, b_{12})$ \\\cline{3-4} 
      & H & 
      $(a_{21}, b_{21})$& 
      $(a_{22}, b_{22})$ \\\cline{3-4}
    \end{tabular}
    \caption {Payoffs for a two player simultaneous move game.}    \label{tab:payoffex}
    \end{center}}
  \end{table}
 \begin{proof}[Proof of Lemma \ref{lemma:pure}]
 The conditions under which various pure strategy N.E. exist are listed in Table \ref{tab:payoffex_conditions}. We show the existence of at least one N.E. in all cases other than when  $a_{22}<a_{12}$, $b_{22}>b_{21}$, $b_{11}>b_{12}$ and $a_{11}<a_{21}$ or $a_{22}>a_{12}$, $b_{22}<b_{21}$, $b_{11}<b_{12}$ and  $a_{11}>a_{21}$. We break down our inspection for equilibria into the following cases,
 \begin{enumerate}
  \item  $a_{22} \geq a_{12}$,  $ b_{22} \geq b_{21}$, then (H, H) is a N.E.
  \item  $a_{22} \leq a_{12}$,  $ b_{22} \leq b_{21}$, 
  		\begin{enumerate}
		\item $a_{11} \geq a_{21}$,  $ b_{11} \geq b_{12}$, then (L, L) is a N.E.
		\item $a_{11} \leq a_{21}$,  $ b_{11} \leq b_{12}$, then (L, H) and (H, L) are N.E.
		\item $a_{11} \geq a_{21}$,  $ b_{11} \leq b_{12}$, then (L, H) is a N.E.
		\item $a_{11} \leq a_{21}$,  $ b_{11} \geq b_{12}$, then (H, L) is a N.E.
		\end{enumerate}
  \item $a_{22} \leq a_{12}$,  $ b_{22} \geq b_{21}$, 
    		\begin{enumerate}
		\item $a_{11} \geq a_{21}$,  $ b_{11} \geq b_{12}$, then (L, L) is a N.E.
		\item $a_{11} \leq a_{21}$,  $ b_{11} \leq b_{12}$, then (L, H) is a  N.E.
		\item $a_{11} \geq a_{21}$,  $ b_{11} \leq b_{12}$, then (L, H) is a N.E.
		\item $a_{11} \leq a_{21}$,  $ b_{11} \geq b_{12}$, then it can be verified if any amongst $a_{22} \leq a_{12}$,  $ b_{22} \geq b_{21}$, $a_{11} \leq a_{21}$ or  $ b_{11} \geq b_{12}$ admits an equality then a N.E. exists. The condition $a_{22}<a_{12}$, $b_{22}>b_{21}$, $b_{11}>b_{12}$ and $a_{11}<a_{21}$ in the lemma describes the condition under which none of the equations admit an equality. 
		\end{enumerate}
   \item $a_{22} \geq a_{12}$,  $ b_{22} \leq b_{21}$ -- This case can be argued as above.
  \end{enumerate}
  Thus if the payoffs are such that $a_{22}<a_{12}$, $b_{22}>b_{21}$, $b_{11}>b_{12}$ and $a_{11}<a_{21}$ or $a_{22}>a_{12}$, $b_{22}<b_{21}$, $b_{11}<b_{12}$ and  $a_{11}>a_{21}$ never simultaneously hold, then there always exists a pure strategy N.E.
       \begin{table}[h]
    {
  \begin{center}
    \setlength{\extrarowheight}{2pt}
    \begin{tabular}{cc|c|c|}
      & \multicolumn{1}{c}{} & \multicolumn{2}{c}{}\\
      & \multicolumn{1}{c}{} & \multicolumn{1}{c}{L}  & \multicolumn{1}{c}{H} \\\cline{3-4}
      \multirow{2}*{}  & L & 
      $ a_{11} \geq a_{21}$,  $ b_{11} \geq b_{12}$& 
      $ a_{22} \leq a_{12}$,  $ b_{11} \leq b_{12}$ \\\cline{3-4} 
      & H & 
      $ a_{11} \leq a_{21}$,  $ b_{22} \leq b_{21}$ & 
      $  a_{22} \geq a_{12}$,  $ b_{22} \geq b_{21}$ \\\cline{3-4}
    \end{tabular}
    \caption {Conditions for existence of N.E. for game with payoffs in Table (\ref{tab:payoffex}).}    \label{tab:payoffex_conditions}
    \end{center}}
  \end{table}
\end{proof}

\begin{proof}[of Theorem~\ref{thm:revplpbph}]
  We show that all of the following four inequalities cannot hold simultaneously,
  \begin{align}
  \frac{1}{f}   &<    \min \{ \frac{1}{1-\beta}, \frac{1}{\rc}, \frac{1}{\rho\beta},  \frac{1}{\rho \rw} \}\label{eq:o1}\\
   \frac{1}{f}  &> \min \{ \frac{1}{\beta}, \frac{1}{\rc}, \frac{1}{\rho(1-\beta)},  \frac{1}{\rho \rw} \},\label{eq:o2}\\
   f& > \frac{\rho}{1-\beta},\label{eq:o3}\\
   f &< \frac{\rho}{\beta}\label{eq:o4}.
  \end{align}
  Combining (\ref{eq:o3}) and (\ref{eq:o4}), we get 
   \begin{align}
  \beta<1-\beta. \label{eq:o5}
    \end{align} 
  Also, combining (\ref{eq:o3}) with the fact that $\beta \in (0,1)$, we have 
   \begin{align}
   (1-\beta)\rho<\rho<\dfrac{\rho}{1-\beta}<&f, \label{eq:o6}\\
      (1-\beta)\rho<&f. \label{eq:o6b}
     \end{align}
     For  (\ref{eq:o1}) and (\ref{eq:o2}) to simultaneously hold, atleast one of the following  must be true,
 \begin{align}
 \dfrac{1}{\beta} &< \dfrac{1}{f}<\min \{ \frac{1}{1-\beta}, \frac{1}{\rc}, \frac{1}{\rho\beta},  \frac{1}{\rho \rw} \},\label{eq:o7}\\
  \dfrac{1}{\rho(1-\beta)} &< \dfrac{1}{f}<\min \{ \frac{1}{1-\beta}, \frac{1}{\rc}, \frac{1}{\rho\beta},  \frac{1}{\rho \rw} \}\label{eq:o8}.
 \end{align}
 Observe that (\ref{eq:o7}) cannot hold as it violates (\ref{eq:o5}). Also, note that (\ref{eq:o8}) cannot hold as it violates (\ref{eq:o6b}). Thus, (\ref{eq:o1}), (\ref{eq:o2}), (\ref{eq:o3}) and (\ref{eq:o4}) cannot hold simultaneously. Arguing on similar lines, we can show that  (\ref{eq:o1}) -- (\ref{eq:o4}) with $\beta$ replaced by $1-\beta$ cannot hold simultaneously as well. Thus, using Lemma \ref{lemma:pure}, there always exists a N.E. for this case.
\end{proof}
\bibliographystyle{spmpsci}      

%
%

\bibliography{bibliography2} 
\end{document}